\documentclass{amsart}
\usepackage[T1]{fontenc}
\usepackage[english]{babel}
\usepackage{amsmath,amsthm,amssymb, latexsym,geometry,stackrel,enumerate}
\usepackage[all]{xy}
\usepackage[dvips]{graphicx}
\DeclareGraphicsExtensions{.bmp,.png,.pdf,.jpgg}
\usepackage[all]{xy}
\usepackage{verbatim}
\usepackage[mathcal]{euscript}
\usepackage{epsfig}
\usepackage{multicol}
\usepackage{color}
\usepackage{verbatim}
\usepackage{graphicx,float}

\swapnumbers
\theoremstyle{plain}
\newtheorem{teo}{Theorem}[section]
\newtheorem{lema}[teo]{Lemma}
\newtheorem{prop}[teo]{Proposition}
\newtheorem{coro}[teo]{Corollary}

\theoremstyle{definition}
\newtheorem{defi}[teo]{Definition}
\newtheorem{obs}[teo]{Remark}

\newtheorem{ejem}[teo]{Example}

\newcommand{\ts}[1]{\normalfont{\textsf{#1}}}
\newcommand{\K}{\ts k}

\renewcommand{\a}{\alpha}
\renewcommand{\b}{\beta}
\renewcommand{\c}{\gamma}

\renewcommand{\d}{ \delta}

\def\D{\mathop{\Delta}\nolimits}

\def\dim{\mathop{\rm dim}\nolimits}

\def\MM{\mathcal{M}}

\newcommand{\A}{\mathcal{A}}
\newcommand{\Fr}[1]{\operatorname{Frobdim}(#1)}

\newcommand{\mor}[3]{$#1\colon #2 \to #3$}

\title[ Open Frobenius Cluster Tilted algebras ]{Open Frobenius Cluster Tilted algebras}

\author[V. Gubitosi]{Viviana Gubitosi}
\address{Instituto de Matem\'{a}tica y Estad\'{\i}stica Rafael Laguardia, Facultad de Ingenier\'{\i}a - UdelaR, Montevideo, Uruguay, 11200 }

\email{gubitosi@fing.edu.uy}



\begin{document}
\maketitle

\begin{abstract}
 In this paper, we compute the Frobenius dimension of any cluster tilted algebra of finite type. Moreover, we give conditions on the bounded quiver of a cluster tilted algebra $\Lambda$  such that $\Lambda$ has non-trivial open Frobenius structures. 
\end{abstract}

\section*{Introduction}

Originally, the term Frobenius algebra referred to an algebra $A$ with
the property that $A\simeq A^*$ as right $A$-modules. Later, Nakayama provided many
equivalent definitions of Frobenius algebras. One equivalent definition, motivated by topological considerations, defines a Frobenius
algebra as an algebra $A$ equipped with a coalgebra structure where the
comultiplication is a map of $A$-modules. This topologically motivated definition
arose in order to rigorously establish the theorem that a two-dimensional topological
quantum field theory is essentially the same as a commutative Frobenius
algebra. The Frobenius algebras  without counit are called \textit{open Frobenius algebras}.
The open Frobenius algebras (also knew as nearly Frobenius algebras \cite{GLS}) give a natural generalization of Frobenius algebras. They were considered for the first time in 2004 by R. Cohen et V. Godin while they were studying the topological quantum field theory of the loop space of a closed
oriented manifold \cite{CG04}. Their main result states that the homology of the free loop space of a closed
oriented manifold is an open Frobenius algebra.  Although open Frobenius algebras generalizes Frobenius algebras, they behave quite differently.

In  \cite{GLS} the authors proved that the direct sums, tensor and quotient of open Frobenius algebras admit  natural
open Frobenius structures. They showed that the family of  open Frobenius structures over an algebra is a $\K$-vector space and they defined  the Frobenius dimension (or $\rm Frobdim$ for short) of an algebra as the dimension of this vector space. Later, in \cite{AGL} an algorithm to  compute the Frobenius dimension  for gentle algebras without oriented cycles is found.

Recently, the Hochschild homology and cohomology of open Frobenius algebras, and the algebraic structures that they have, have been studied by H. Abbaspour \cite{A15}. In particular, the Hochschild homology of an open Frobenius algebra has a natural structure of open Frobenius algebra, which provides a way to produce new open Frobenius algebras  from the already known ones. Another example of open Frobenius algebra is  the Hochschild cohomology $\rm HH^*(A)$ of a closed Frobenius algebra $A$ \cite{TZ06}.

 On the other hand, cluster categories were introduced in  \cite{BMRRT} as a representation theoretic framework for the  cluster algebras of Fomin and Zelevinski \cite{FZ02}. Given an hereditary finite dimensional algebra $H$  over  an algebraically closed field $\K$ the cluster category is defined to be $\mathcal{C}(H):=\mathcal{D}^b(H)/ \tau^{-1} [1]$, where $[1]$  denotes the  shift functor  and $\tau$ is the Auslander - Reiten translation in $\mathcal{D}^b(H)$. By a result of Keller \cite{K05}, the cluster category is triangulated.  For the  cluster category,  cluster tilting objects have been defined also  in \cite{BMRRT}, where in addition the authors showed that the clusters correspond to the tilting objects in the cluster category. The endomorphism algebras of the cluster tilting objects are called cluster-tilted algebras and were introduced by Buan, Marsh and Reiten \cite{BMR} and, independently in \cite{CCS} for type $\mathbb{A}$. Since then, cluster-tilted algebras have been the subject of several investigations, see, for instance, \cite{BMR2,ABS,ABS2,ABS3,ABS4,BV,BHL,V}. In \cite{BMR2} the authors gave an explicit description for the quivers of cluster-tilted algebras of finite representation type and as a consequence they showed that a (basic) cluster-tilted algebra of finite type is uniquely determined by its quiver. In \cite{BV} Buan and Vatne gave a criterion to decide whether two cluster-tilted algebras of type $\mathbb{A}$ are themselves derived equivalent or not. For type $\mathbb{E}$ the same work was done by Bastian, Holm and Ladkani in \cite{BHL}. For type $\mathbb{D}$ Vatne gave a complete description of the cluster-tilted algebras \cite{V}.

Our purpose in this paper is to study the open (or nearly) Frobenius structures that a cluster-tilted algebra of finite type can admit. This include examples of gentle algebras having  oriented cycles which generalize the main result of \cite{AGL}.

We now state the main results of this paper (for the definitions of the terms used, we refer the reader to sections 1 and 2  below).

\subsection*{Theorem A}\textit{If $\A$ is a  cluster tilted algebra of type $\mathbb{A}$, then $\A$  has  finite Frobenius dimension. Moreover,
$$\Fr\A= \sharp \mathcal{B} + \displaystyle{\sum_{b\in \mathcal{V_S}} \ell_{\rightarrow b}.\ell_{\leftarrow b} } $$}

\subsection*{Theorem B}\textit{If $\A$ is a  cluster tilted algebra of type $\mathbb{D}$, then $\A$  has  finite Frobenius dimension.}

\subsection*{Theorem C}\textit{If $\A$ is a non hereditary cluster tilted algebra of type $\mathbb{E}_6$, then $\Fr\A\geq 1$.}

\vspace{.5cm}

In particular we found a  large number of examples of algebras  that can be endowed with open Frobenius structures.

\vspace{.5cm}

The paper is organized as follows: In section 1 we recall facts about quivers and algebras, cluster-tilted algebras and open Frobenius  algebras. We use this section to fix some notation. Also  we establish the facts about cluster tilted algebras of finite type  that will be used in the sequel. In section 2 we compute the Frobenius dimension for cluster tilted algebra of type $\mathbb{A}$. We finish section 2 with some consequences, among which, we show that the Frobenius dimension is not invariant under sink/source
mutations and the fact of have non-zero Frobenius dimension neither. Section 3 is devoted to compute the Frobenius dimension for cluster tilted algebras of type $\mathbb{D}$ or find  a lower bound. Finally, in section 4 we prove that cluster tilted algebras of type $\mathbb{E}_6$ admit at least one non-trivial open  Frobenius structure.


\section{Preliminaries}

\subsection{Quivers and  Algebras}

While we briefly recall some  concepts concerning bound quivers and algebras, we refer the reader to \cite{ASS06} or \cite{ARS95}, for instance, for unexplained notions.

Let $\K$   be an algebraically closed field. A quiver $Q$ is the data of two sets, $Q_0$ (the \textit{vertices}) and $Q_1$ (the \textit{arrows}) and two maps \mor{s,t}{Q_1}{Q_0} that assign to each arrow $\a$ its \textit{source} $s(\a)$ and its \textit{target} $t(\a)$. We write \mor{\a}{s(\a)}{t(\a)}. If $\b\in Q_1$ is such that $t(\a)=s(\b)$ then the composition of $\a$ and $\b$ is the path $\a\b$. This extends naturally to paths of arbitrary positive length. The \emph{path algebra} $\K Q$ is the $\K$-algebra whose basis is the set of all paths in $Q$, including one stationary path $e_x$ at each vertex $x\in Q_0$, endowed with the  multiplication induced from the composition of paths. In case $|Q_0|$ is finite, the sum of the stationary paths  - one for each vertex - is the identity.

If the quiver $Q$ has no oriented cycles, it is called \emph{acyclic}. A \emph{relation} in $Q$ is a $\K$-linear combination of paths of length at least $2$ sharing source and target.  A relation which is a path is called \emph{monomial}, and the relation is \emph{quadratic} if the paths appearing in it have all length $2$. Let $\mathcal{R}$ be a set of relations.
 Given $\mathcal{R}$ one can consider the two-sided ideal of $\K Q$ it generates $I=\langle \mathcal{R}\rangle \subseteq  \langle Q_1 \rangle^2$. It is called \emph{admissible} if there exists a natural number $r\geqslant 2$ such that $\left\langle Q_1 \right\rangle^r \subseteq I$. The pair $(Q,I)$ is a \emph{bound quiver}, and associated to it is the algebra $A=\K Q/I$.
It is known that any finite dimensional basic algebra over an algebraically closed field is obtained in this way, see \cite{ASS06}, for instance.\\

We are interested in a particular family of path algebras, called cluster-tilted algebras.

\subsection{Cluster-tilted Algebras}

Let $\K$ be an algebraically closed field. We consider connected hereditary finite dimensional $\K$-algebras. Any such algebra $ H$ is Morita equivalent to a path algebra $\K Q$ for some finite quiver $Q$. Furthermore, we assume $H$ is of finite representation type, that is, there is only a finite number of indecomposable objects, up to isomorphism, in the category $\rm mod$ $H$ of finitely  generated $H$-modules. It is well know that this is the case if and only if the underlying graph of $Q$ is a Dynkin graph $\mathbb{A},\mathbb{D}$ or $\mathbb{E}$.

Let $\mathcal{D}$ be the bounded derived category $\mathcal{D}^b(H)$. It is equipped with a shift functor $[1]$ and a translation functor $\tau$ with quasi-inverse $\tau^{-1}$. The composition $F=\tau^{-1}[1]$ is an auto-equivalence in $\mathcal{D}$. Then we can define the \textit{cluster category} as the orbit category  $\mathcal{C}:=\mathcal{D}/ F$. An objet $T$ of  $\mathcal{C}$ is called a tilting object if $\rm Ext^1_{\mathcal{C}}(T,T)=0$ and $T$ is maximal with respect to this property. The endomorphism algebra $\rm End_{\mathcal{C}}(T)$ is called a \textit{cluster-tilted algebra}.  See \cite{BMRRT,BMR} for more details.

When $Q$ is a Dynkin quiver of types $\mathbb{A}$, $\mathbb{D}$ or $\mathbb{E}$, the
corresponding cluster-tilted algebras are said to be of Dynkin type.
These algebras have been investigated in \cite{BMR}, where it is
shown that they are schurian and moreover they can be defined by using
only zero and commutativity relations that can be extracted from their
quivers in an algorithmic way. The possible quivers are precisely the quivers in the mutation class
of the Dynkin quiver. By a result of Fomin and
Zelevinsky~\cite{FominZelevinsky03}, the mutation class of a Dynkin quiver is finite. Moreover, the quivers in the mutation classes of Dynkin quivers are
explicitly known; for type $\mathbb{A}$ they can be found in~\cite{BV}, for type $\mathbb{D}$ in~\cite{V} and for type $\mathbb{E}$ they can be enumerated using a
computer, for example by the Java applet~\cite{Keller-software}.

\subsection{Open Frobenius algebras}
The Frobenius algebras without counit are called open  Frobenius algebras. They
 give a natural generalization of Frobenius algebras. We briefly recall some  concepts concerning open Frobenius algebras.

An algebra $\A$ with  multiplication $m:\A\otimes \A \rightarrow \A$ is an \emph{open Frobenius algebra} if it admits a linear map $\Delta:\A \rightarrow \A \otimes \A$ such that  $\Delta$ is an  $\mathcal{A}$-bimodule morphism, i.e.  $$\bigl(m\otimes 1\bigr)\bigl(1\otimes\Delta\bigr)=\Delta\circ m=\bigl(1\otimes m\bigl)\bigl(\Delta \otimes 1\bigr)$$

Observe that any open Frobenius coproduct in $\A$ is determined by its value on  the unit  $1$ of $\A$.  That is,  if $\Delta:\A \rightarrow \A \otimes \A$ is an open Frobenius coproduct, then $\Delta(x)=\bigl(x\otimes 1\bigr)\Delta(1)=\Delta(1)\bigl(1\otimes x\bigr)$ for all $x\in \A$.

If $\A=\K Q/I$ is a path algebra, the stationary paths $e_x$ with  $x\in Q_0$ are idempotents, then $\Delta(e_x)=\bigl(e_x\otimes 1\bigr)\Delta(e_x)=\Delta(e_x)\bigl(1\otimes e_x\bigr)$.
If in addition $|Q_0|$ is finite the sum $\sum _{x\in Q_0} e_x=1$. \\


Following \cite{GLS} the \emph{Frobenius space} associated to an algebra $\A$ is the $\K$-vector space $\mathcal{E}$ of all  possible coproducts $\Delta$ that make $\A$ into an open Frobenius algebra. Its dimension over $\K$ is called the \emph{Frobenius dimension} of $\A$, i.e, $\rm Frobdim(\A)=\dim_{\K}(\mathcal{E})$.


\section{Open Frobenius Cluster-Tilted Algebras of type $\mathbb{A}$}

Let $\MM^A_k$ be the mutation class of $\mathbb{A}_k$. This is the  set of all  connected quivers with $k$ vertices that satisfy the following \cite{BV}:
\begin{itemize}
\item all non-trivial cycles are  3-cycles,
\item every vertex has valency at most four,
\item if a vertex has valency four, then two of its adjacent arrows belong
  to one 3-cycle, and the other two belong to another 3-cycle,
\item if a vertex has valency three, then two of its adjacent arrows belong to
  a 3-cycle, and the third arrow does not belong to any 3-cycle.
\end{itemize}

By a \emph{3-cycle}  we mean an oriented  cycle of length 3.  Then $A=\K Q/ I $ is a cluster-tilted algebra of type $\mathbb{A}_k$ if and only if   $Q\in \MM^A_k$ and every 3-cycle is \textit{saturated} (i.e, the composition of two consecutive arrows belong to the ideal $I$) \cite{BMR}. \\


\begin{defi} Given $Q\in \MM^A_k$ a subquiver of $Q$ is said to be a \emph{tail}  if is a Dynkin diagram of type $\mathbb{A}_n$ with $n \geq 2$ maximal and any orientation.
\end{defi}



Throughout this section $Q$  is a quiver in $\MM^A_k$, every 3-cycle is saturated and $\Delta$ is any open Frobenius structure for the cluster tilted algebra associated.

A tail with the orientation $\SelectTips{eu}{10}\xymatrix@R=.8pc@C=.5pc{ \scriptstyle{v_1} \ar[rr]^{\a_1} & & \scriptstyle{v_2} \ar[rr]^{\a_2} & & \scriptstyle{v_3}  \ar@{.}[rr] && \scriptstyle{v_{n-1}}  \ar[rr]^{\a_{n-1}}  &&  \scriptstyle{v_{n}} }$ will be called a \emph{lineal tail}. According to \cite[Lemma 1]{AGL} if $Q$ is a lineal tail, $A=\K Q$  admits only one open Frobenius structure (called \emph{the lineal structure}) given by:

\begin{eqnarray*}
\D(e_{v_1})& = & \a_1\cdots\a_{n-1}\otimes e_{v_1}\\
\D(e_{v_i})& = & \a_i\cdots\a_{n-1}\otimes \a_1\cdots\a_{i-1}\\
\D(e_{v_n})& = & e_{v_n}\otimes \a_1\cdots\a_{n-1}
\end{eqnarray*}

\begin{lema}\label{tail-con-vertice-de-valencia-dos}
Let $\mathcal{T}$ be a tail of $Q$ having a source or a sink  of valency two.  Then $\D(e_v)=0$   for all $v\in \mathcal{T}_0$.
\end{lema}

\begin{proof}
It is a direct consequence of \cite[Lemma 2]{AGL}.
\end{proof}

Summarizing we have the following  result for hereditary cluster tilted algebras of type $\mathbb{A}$.\

\begin{prop}
An hereditary cluster tilted algebra $A=\K Q$ of type $\mathbb{A}$ has non zero Frobenius dimension if and only if  $Q$ is a  linearly oriented tail.
\end{prop}

\begin{proof} It is enough to observe that the quiver of an  hereditary cluster tilted algebras of type $\mathbb{A}$ is  a tail. Then the result follows from \cite[Lemma 1]{AGL} and previous lemma. 
\end{proof}

Now, we are interested in compute the Frobenius dimension for non hereditary cluster tilted algebras of type $\mathbb{A}$. \

\begin{defi}\label{def de camino base y vertice especial} Given $A=\K Q/ I $ we say that
\begin{itemize}
  \item [a)] a vertex $v\in Q_0$ is an \emph{special vertex} if $v$ has valency two and there are $\alpha, \beta \in Q_1$ with $s(\b)=t(\a)=v$ and $\a\b\in I$. Let denote by $\mathcal{V_S}$ the set of all special vertices.
  \item [b)] a \emph{basis path }  is a path $\mathcal{P} : v \rightsquigarrow v'$ such that $v,v'$ are or an special vertex  or a source or a sink of valency one. We denote by   $\mathcal{B}$ the set of all basis paths.
\end{itemize}
\end{defi}

The following figure show  the three types of basis paths $\mathcal{P} : 1 \rightsquigarrow n$.

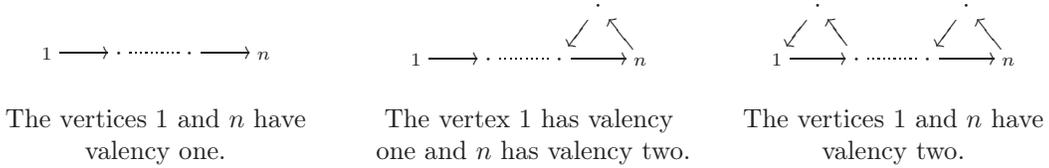
\begin{figure}[H]
$$\begin{array}{ccccc}
  \SelectTips{eu}{10}\xymatrix@R=.8pc@C=.5pc{&&&&&&\\
\scriptstyle{1}   \ar[rr] & & \cdot  \ar@{.}[rr] && \cdot  \ar[rr]  && \scriptstyle{n}  } & & \SelectTips{eu}{10}\xymatrix@R=.8pc@C=.5pc{&& && & \cdot \ar[dl]& \\
\scriptstyle{1}   \ar[rr] & & \cdot  \ar@{.}[rr] && \cdot  \ar[rr]  && \scriptstyle{n} \ar[lu] } & & \SelectTips{eu}{10}\xymatrix@R=.8pc@C=.5pc{& \cdot \ar[dl]& && & \cdot \ar[dl]& \\
\scriptstyle{1}   \ar[rr] & & \cdot \ar[ul] \ar@{.}[rr] && \cdot  \ar[rr]  && \scriptstyle{n} \ar[lu] } \\
   & &&  &  \\
  \text{ The vertices $1$ and $n$ have } && \text{ The vertex $1$ has valency }  && \text{The vertices $1$ and $n$ have} \\
\text{valency one.}&& \text{ one and  $n$ has valency two.}&& \text{valency two.}
\end{array}$$
\caption{Types of basis paths}
\end{figure}

Given a vertex $b$ in a quiver $Q$ we will use the symbol $\ell_{\rightarrow b}$ to denote the length of the largest path with target $b$ and $\ell_{\leftarrow b}$ the length of the largest path with source  $b$.\\

We start describing  a coproduct $\D$ over the vertices of an arbitrary basis path of any type and computing the corresponding Frobenius dimension.

\begin{lema}\label{lineal-freecycle}
Let $\A$ be the algebra given by the bound quiver

$$\SelectTips{eu}{10}\xymatrix@R=.8pc@C=.5pc{&&&&&&&& & \scriptstyle{b} \ar[rd]^{\gamma_2}  & &&& \\
 \scriptstyle{1} \ar[rr]_{\a_1}& & \scriptstyle{2}  \ar[rr]_{\a_2} &&   \scriptstyle{3}  \ar@{.}[rr]& &  \scriptstyle{n}  \ar[rr]_{\a_{n}}  && \scriptstyle{a} \ar[ru]^{\gamma_1}  &&  \scriptstyle{c} \ar[ll]^{\gamma_3} }$$

Then $\Fr\A=2+\ell_{\rightarrow b}\ell_{\leftarrow b}+\ell_{\rightarrow c}\ell_{\leftarrow c}$.
\end{lema}

\begin{proof}
Every idempotent $e_x$ satisfies  $\Delta(e_x)=\bigl(e_x\otimes 1\bigr)\Delta(e_x)\bigl(1\otimes e_x\bigr)$. Then, the coproduct in $e_1$  and $e_b$ are  given by:
 \hspace{.5cm} $\D(e_1)=A  e_1\otimes e_1+\sum_{i=1}^n B_i  \a_1\cdots\a_i\otimes e_1+F \a_1\cdots\a_n\c_1\otimes e_1$  \hspace{.5cm} and \newline  $\D(e_b)=A'  e_b\otimes e_b +B \c_2\otimes e_b +Ce_b\otimes \c_1+\sum_{i=1}^n C_i  e_b\otimes\a_i\cdots\a_n\c_1 +\sum_{i=1}^n D_i  \c_2\otimes\a_i\cdots\a_n\c_1+D\c_2\otimes \c_1 .$\\


Using that $\D(\a_1\cdots\a_n\c_1)=\D(e_1\a_1\cdots\a_n\c_1)=\D(\a_1\cdots\a_n\c_1e_b)$ and $\c_1\c_2=0$ we deduce that $F=C_1$ and $A=A'=C=B_i=C_j=0$ for  $i=1,\cdots, n$ and  $j=2,\cdots, n$. Therefore
\begin{eqnarray*}
\D(e_1)&=&F \a_1\cdots\a_n\c_1\otimes e_1\\
\D(e_b)&=& B \c_2\otimes e_b + F  e_b\otimes\a_1\cdots\a_n\c_1 +\sum_{i=1}^n D_i  \c_2\otimes\a_i\cdots\a_n\c_1+D\c_2\otimes \c_1
\end{eqnarray*}

A general expression for $\D(e_i)$ is $\D(e_i)= \sum_{\rho,\eta} G_{\rho,\eta} \rho \otimes \eta$  with $s(\rho)=t(\eta)=i$. Using that $\D(\a_1\cdots \a_i)=\D(e_1\a_1\cdots \a_i)=\D(\a_1\cdots \a_ie_i)$ we deduce that $\D(e_i)=F \a_i\cdots\a_n\c_1\otimes \a_1\cdots \a_{i-1}$ for $i=2,\cdots ,n$.
Similarly, $\D(e_a)=F \c_1\otimes \a_1\cdots\a_n$.

Finally $\D(e_c)= \sum_{\rho,\eta} G_{\rho,\eta} \rho \otimes \eta$  with $s(\rho)=t(\eta)=c$. Comparing the expressions $\D(\c_2e_c)$ and $\D(e_b\c_2)$ we obtain that $\D(e_c)=B e_c\otimes \c_2 +E \c_3\otimes \c_2$.\

Observe that $\ell_{\rightarrow c}\ell_{\leftarrow c}=1=\sharp \{E\}$, $ \ell_{\rightarrow b}\ell_{\leftarrow b}= \sharp \{D_i,D\}$ and we have two independent linear structures associated to the scalars $B$ and $F$ which completes the proof.
\end{proof}

\begin{coro}\label{tail con fuente/pozo + ciclo}
Let $\A$ be an algebra as above such that in the tail $\mathcal{T}$: $\SelectTips{eu}{10}\xymatrix@R=.4pc@C=.4pc{
 \scriptstyle{1} \ar@{-}[rr]^{\a_1}& & \scriptstyle{2}  \ar@{-}[rr]^{\a_2} &&   \scriptstyle{3}  \ar@{.}[rr]& &  \scriptstyle{n}  \ar@{-}[rr]^{\a_{n}}  && \scriptstyle{a}}$ there is a vertex $2\leq i\leq n$  which is a source or a sink. Then $\D_{\mid _{\mathcal{T}}}\equiv 0$ and $\Fr\A=1+\ell_{\rightarrow b}+\ell_{\leftarrow c}$.
\end{coro}


\begin{proof}
From lemma \ref{tail-con-vertice-de-valencia-dos} it is clear that $\D(e_i)=0$ for    $i=1,\cdots ,n$. The coproduct $\D$ could change for the vertices $a,b$ and $c$. There is no loss of generality in assuming $\SelectTips{eu}{10}\xymatrix@R=.6pc@C=.6pc{\a_n:n \ar[r]&  a}$, then $0=\D(e_n\a_n)=\D(\a_n)=\D(\a_n e_a)$ implies $\D(e_a)=0$ because $\a_n\mu\neq 0$ for all $\mu$ starting at $a$. Similarly we can see that $\D(e_c)$ does not change and $\D(e_b)= B \c_2\otimes e_b +  \displaystyle{\sum_{\mu / t(\mu)=a} D_{\mu}  \c_2\otimes\mu\c_1}$ ; with  $\mu\neq \c_3$.
\end{proof}


\begin{coro}\label{lineal-freecycle-lineal}
Let $\A$ be the algebra given by the bound quiver

$$\SelectTips{eu}{10}\xymatrix@R=.8pc@C=.5pc{&&&&&&&& & \scriptstyle{b} \ar[rd]^{\gamma_2}  & &&& \\
 \scriptstyle{1} \ar[rr]_{\a_1}& & \scriptstyle{2}  \ar[rr]_{\a_2} &&   \scriptstyle{3}  \ar@{.}[rr]& &  \scriptstyle{n}  \ar[rr]_{\a_{n}}  && \scriptstyle{a} \ar[ru]^{\gamma_1}  &&  \scriptstyle{c} \ar[ll]^{\gamma_3} \ar[rr]_{\delta_1}  &&  \scriptstyle{n+1}    \ar@{.}[rr] &&  \scriptstyle{n+m-1} \ar[rr]_{\delta_m} && \scriptstyle{n+m} } $$
 Then  $\Fr\A= 2 + \ell_{\rightarrow b}\ell_{\leftarrow b}$.
\end{coro}

\begin{proof}
Analysis similar to that in the proof of lemma \ref{lineal-freecycle}  shows that
\begin{eqnarray*}
\D(e_1) & = & F \a_1\cdots\a_n\c_1\otimes e_1\\
\D(e_i) & = & F \a_i\cdots\a_n\c_1\otimes \a_1\cdots \a_{i-1}  \text{; \ \ for  }  i=2,\cdots ,n. \\
\D(e_a) &=& F \c_1\otimes \a_1\cdots\a_n \\
\D(e_b) &=& B \c_2 \d_1\cdots \d_m \otimes e_b + F  e_b\otimes\a_1\cdots\a_n\c_1 +  \displaystyle{\sum_{\mu / t(\mu)=a \ \ }} \displaystyle{\sum_{\eta / s(\eta)=c} D_{\mu,\eta}  \c_2 \eta \otimes \mu\c_1} ;  \text{  \ \ with $\mu,\eta\neq \c_3 $ }\\
\D(e_c) &=& B \d_1\cdots \d_m \otimes \c_2  \\
\D(e_{n+j}) & = & B d_{j+1}\cdots \d_{m} \otimes  c_2\d_1\cdots\d_j \text{;  \ \ for  }  j=1,\cdots ,m-1. \\
\D(e_{n+m}) & = & B \c_2\d_1\cdots\d_{m}\otimes  e_{n+m}
\end{eqnarray*}
 Then we have  a linear structure for both basis path  $\a_1\cdots\a_n\c_1$ and  $\c_2 \d_1\cdots \d_m $ (associated to the scalars $F$ and $B$ respectively). The special vertex $b$ add $\sharp \{ D_{\mu,\eta}\}= \ell_{\rightarrow b}\ell_{\leftarrow b}$.
\end{proof}

\begin{lema}\label{caso-particular-freecycle-lineal-freecycle}
Let $\A$ be the algebra given by the bound quiver

$$\SelectTips{eu}{10}
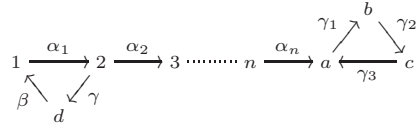
\xymatrix@R=.8pc@C=.5pc{&&&&&&& &&\scriptstyle{b} \ar[rd]^{\gamma_2}  & &&& \\
 \scriptstyle{1} \ar[rr]^{\a_1}& & \scriptstyle{2}   \ar[dl]^{\gamma} \ar[rr]^{\a_2} &&   \scriptstyle{3}\ar@{.}[rr] & &  \scriptstyle{n}  \ar[rr]^{\a_{n}}  && \scriptstyle{a} \ar[ru]^{\gamma_1}  &&  \scriptstyle{c} \ar[ll]^{\gamma_3} \\
 & \scriptstyle{d} \ar[lu]^{\b}  & &&&&&& }$$

Then $\Fr\A=3+\ell_{\rightarrow b}\ell_{\leftarrow b}+\ell_{\rightarrow 1}\ell_{\leftarrow 1}+\ell_{\rightarrow c}\ell_{\leftarrow c}+\ell_{\rightarrow d}\ell_{\leftarrow d}$.
\end{lema}


\begin{proof}
First observe that if we delete the arrows $\b$ and $\c$ we get an algebra like the lemma above. Clearly $\D$ only can change for the vertices $e_1$ and $e_2$. Then

$$\begin{array}{lcl}
\D(e_i)&=&F \a_i\cdots\a_n\c_1\otimes \a_1\cdots \a_{i-1}, \text{ for } i=3,\cdots ,n.\\
\D(e_a)&=&F \c_1\otimes \a_1\cdots\a_n\\
\D(e_b)&=& B \c_2\otimes e_b + F  e_b\otimes\a_1\cdots\a_n\c_1 +\sum_{i=1}^n D_i  \c_2\otimes\a_i\cdots\a_n\c_1+E\c_2\otimes \c_1\\
\D(e_c)&=&B e_c\otimes \c_2 +D \c_3\otimes \c_2\\
\end{array}$$

The equality $\D(e_2\a_2)=\D(\a_2e_3)=F \a_2\a_3\cdots\a_n\c_1\otimes \a_1\a_2$ implies $\D(e_2)=F \a_2\cdots\a_n\c_1\otimes \a_1$.
Similarly, comparing $\D(e_1)$ and $\D(e_b)$ using the path $\a_1\cdots\a_n\c_1$ we obtain that

$$\D(e_1)=F \a_1\cdots\a_n\c_1\otimes e_1+\sum_{i=1}^n D'_i \a_1\cdots\a_i \otimes\b + I\a_1\cdots\a_n\c_1\otimes \b +H e_1\otimes \b .$$

A general expression for  $\D(e_d)= J\b\otimes \c +K \b\otimes e_d +L e_d\otimes\c$. Comparing with $\D(e_2)$ using the arrow $\c$ we obtain $L=0$ and comparing with $\D(e_1)$ using the arrow $\b$ we get $H=K$. Finally  $\D(e_d)= J\b\otimes \c +H \b\otimes e_d$. \
 We conclude that $\Fr\A=5+2(n+1) = 3+\ell_{\rightarrow b}\ell_{\leftarrow b}+\ell_{\rightarrow 1}\ell_{\leftarrow 1}+\ell_{\rightarrow c}\ell_{\leftarrow c}+\ell_{\rightarrow d}\ell_{\leftarrow d}$.
\end{proof}

\begin{coro}
Let $\A$ be an algebra as above such that  the tail $\mathcal{T}$: $\SelectTips{eu}{10}\xymatrix@R=.4pc@C=.4pc{
  \scriptstyle{2}  \ar@{-}[rr]^{\a_2} &&   \scriptstyle{3}  \ar@{.}[rr]& &  \scriptstyle{n}  \ar@{-}[rr]^{\a_{n}}  && \scriptstyle{a}}$ has a vertex $3\leq i\leq n$  which is a source or a sink. Then $\D_{\mid _{\mathcal{T}}}\equiv 0$ and $\Fr\A=\ell_{\rightarrow b}+\ell_{\leftarrow 1}+4$.
\end{coro}

\begin{proof}
Analysis similar to that in the proof of  corollary \ref{tail con fuente/pozo + ciclo} shows that $\D_{\mid _{\mathcal{T}}}\equiv 0$, $\D(e_c)$ and $\D(e_d)$ do not change  and
\begin{eqnarray*}
\D(e_b) &=& B \c_2\otimes e_b +  \displaystyle{\sum_{\mu / t(\mu)=a} D_{\mu}  \c_2\otimes\mu\c_1}  \text{ \ \ ; $\mu \neq \c_3$}\\
\D(e_1) &=& H e_1\otimes \b + \displaystyle{\sum_{\eta / s(\eta)=2} D_{\eta}  \a_1\eta \otimes\b}    \text{ \ \ ; $\eta \neq \c_3$}
\end{eqnarray*}
which completes the proof.

\end{proof}

In some cases we do not have basis paths and we can find the following situation.

\begin{lema}\label{los-tres-tails-salientes}
Let $\A$ be the algebra given by the bound quiver

$\SelectTips{eu}{10}
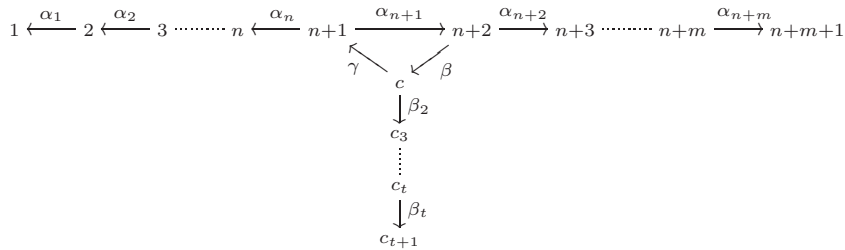
\xymatrix@R=.8pc@C=.5pc{ \scriptstyle{1} & & \scriptstyle{2} \ar[ll]_{\a_1}  &&   \scriptstyle{3} \ar[ll]_{\a_2} & & \ar@{.}[ll] \scriptstyle{n}  && \scriptstyle{n+1}  \ar[ll]_{\a_{n}} \ar[rr]^{\a_{n+1}}  &&  \scriptstyle{n+2} \ar[dl]^{\b} \ar[rr]^{\a_{n+2}} & & \scriptstyle{n+3}  \ar@{.}[rr] && \scriptstyle{n+m}  \ar[rr]^{\a_{n+m}}  && \scriptstyle{n+m+1} \\
&&&&&&&& & \scriptstyle{c} \ar[lu]^{\gamma} \ar[d]^{\b_2} & &&&&&& \\
&&&&&&&& & \scriptstyle{c_3} \ar@{.}[d] & &&&&&& \\
&&&&&&&& & \scriptstyle{c_t} \ar[d]^{\b_t} & &&&&&& \\
&&&&&&&& & \scriptstyle{c_{t+1}}  & &&&&&&}$

Then $\D\equiv0$ and $\Fr\A=0$.
\end{lema}

\begin{proof}
It follows by comparing $\D(e_1)$ with $\D(e_{n+1})$ using the path $\a_n\cdots\a_1$. By symmetry we obtain $\D(e_{n+2})$, $\D(e_{n+m+1})$, $\D(e_c)$ and $\D(e_{c_{t+1}})$. Then,  compare $\D(e_{n+2})$ with $\D(e_{n+1})$ using the arrow $\a_{n+1}$ to find that $\D(e_i)=0$ for all $1\leq i\leq n+m+1$. Finally, conclude that  $\D(e_c)=\D(e_{c_j})=0$ for all $3\leq j\leq t+1$.
\end{proof}

Observe that if $(Q,I)$ is the bound quiver of a cluster tilted algebra of type $\mathbb{A}$ without basis paths and without sources or sinks of valency two, then $(Q,I)$ has  a subquiver as above (or its dual). \\

Now we want to see what happens if we add arrows to a basis path. \\

\begin{defi}
Let $\SelectTips{eu}{10}\xymatrix@R=.8pc@C=.5pc{\scriptstyle{i}\ar[rr]^{\a_i} & &  \scriptstyle{i+1} }$ be an arrow of a quiver  $Q$: $\SelectTips{eu}{10}\xymatrix@R=.8pc@C=.5pc{ \ar@{.}[rr]& & \scriptstyle{i-1} \ar@{-}[rr] & & \scriptstyle{i} \ar[rr]^{\a_i} && \scriptstyle{i+1}  \ar@{-}[rr]  &&   \scriptstyle{i+2} \ar@{.}[rr]  & & }$. We say that $Q'$  is obtained by \emph{adding} to $Q$ \emph{a tail  through the arrow }$\a_i$ if $Q'$ has the form

$$\SelectTips{eu}{10}
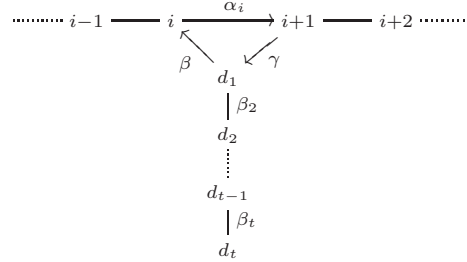
\xymatrix@R=.8pc@C=.5pc{ \ar@{.}[rr]& & \scriptstyle{i-1} \ar@{-}[rr] & & \scriptstyle{i} \ar[rr]^{\a_i} && \scriptstyle{i+1} \ar[dl]^{\gamma}  \ar@{-}[rr]  &&   \scriptstyle{i+2} \ar@{.}[rr]  & & \\
&&&& & \scriptstyle{d_1} \ar[lu]^{\b} \ar@{-}[d]^{\b_2} & &&&&&& \\
&&&& & \scriptstyle{d_2} \ar@{.}[d] & &&&&&& \\
&&&& & \scriptstyle{d_{t-1}} \ar@{-}[d]^{\b_t} & &&&&&& \\
&&&& & \scriptstyle{d_t}  & &&&&&&}$$

with the cycle $\a_i\c\b$ saturated.
\end{defi}

\begin{lema}\label{delta en el tail da cero}
Let $\A$ be the algebra given by  the quiver $Q$:
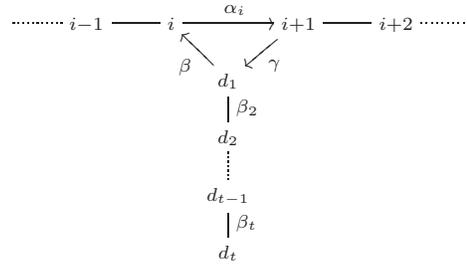
$$\SelectTips{eu}{10}\xymatrix@R=.8pc@C=.5pc{ \ar@{.}[rr]& & \scriptstyle{i-1} \ar@{-}[rr] & & \scriptstyle{i} \ar[rr]^{\a_i} && \scriptstyle{i+1} \ar[dl]^{\gamma}  \ar@{-}[rr]  &&   \scriptstyle{i+2} \ar@{.}[rr]  & & \\
&&&& & \scriptstyle{d_1} \ar[lu]^{\b} \ar@{-}[d]^{\b_2} & &&&&&& \\
&&&& & \scriptstyle{d_2} \ar@{.}[d] & &&&&&& \\
&&&& & \scriptstyle{d_{t-1}} \ar@{-}[d]^{\b_t} & &&&&&& \\
&&&& & \scriptstyle{d_t}  & &&&&&&}$$

with $t\geq 2$. If $\D(\b)=\D(\c)=0$, then $\D(e_{d_i})=0$ for all $ i=1,\cdots, t$.
\end{lema}

\begin{proof}
By \ref{tail-con-vertice-de-valencia-dos} it suffices to consider the case where the tail $\SelectTips{eu}{10}\xymatrix@R=.8pc@C=.5pc{ \scriptstyle{d_1} \ar@{-}[rr]^{\b_2} & & \scriptstyle{d_2} \ar@{.}[rr] & & \scriptstyle{d_{t-1}} \ar@{-}[rr]^{\b_t} && \scriptstyle{t}}$ is linearly oriented. We can assume that there is a path $\b_2\cdots\b_t$ from $d_1$ to $d_t$. Since $\D(\c)=0$ we have  $0=\D(\c\b_2\cdots\b_i)=\D(\c\b_2\cdots\b_i e_{d_i})$  which implies $\D(e_{d_i})=0$ for all $i\neq 1$. For $d_1$ compare  $\D(e_{d_1})$ with $\D(e_{d_2})=0$ using the arrow $\b_2$.

\end{proof}

If an algebra as above has $t=1$ we have the following particular case:
$$\SelectTips{eu}{10}
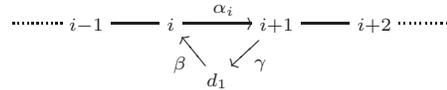
\xymatrix@R=.8pc@C=.5pc{ \ar@{.}[rr]& & \scriptstyle{i-1} \ar@{-}[rr] & & \scriptstyle{i} \ar[rr]^{\a_i} && \scriptstyle{i+1} \ar[dl]^{\gamma}  \ar@{-}[rr]  &&   \scriptstyle{i+2} \ar@{.}[rr]  & & \\
&&&& & \scriptstyle{d_1} \ar[lu]^{\b}  & &&&&&& }$$

\begin{lema}\label{caso particular almost free cycle}
Let $\A$ be the algebra as above with  $\D(\b)=\D(\c)=0$, then $$\D(e_{d_1})= \displaystyle{\sum_{\small{s(\rho)=i}}\sum_{\small{t(\mu)=i+1}} A_{\rho\mu}\b\rho\otimes\mu\c }.$$
In particular, if $\SelectTips{eu}{10}\xymatrix@C=1pc{ \scriptstyle{i-1} \ar[r] &  \scriptstyle{i} \ar[r]^{\a_i} & \scriptstyle{i+1}  \ar[r]  &   \scriptstyle{i+2}  }$ is linearly oriented $\D(e_{d_1})= A \b\otimes\c $.
\end{lema}

\begin{proof}
The coproduct $\D$ at the vertex $e_d$ has the general expression  $\D(e_{d_1})= \displaystyle{\sum_{\sigma,\eta} G_{\sigma,\eta} \sigma \otimes \eta}$  with $s(\sigma)=t(\eta)=d_1$. As  $0=\D(\c)=\D(\c e_{d_1})$ and  $\c\b=0$ only the terms with $\sigma=\b\rho$ where $s(\rho)=i$ have a non zero  $G_{\sigma,\eta}$.  Analogously $0=\D(\b)$ implies that only the terms with $\eta=\mu\c$  where $t(\mu)=i+1$ have a non zero  $G_{\sigma,\eta}=A_{\rho\mu}$.
\end{proof}

\begin{obs}\label{vertices especiales suman l>--.l-->}  Since  $\sharp \{A_{\rho\mu} : s(\rho)=i, t(\mu)=i+1\}= \ell_{\rightarrow d_1}\ell_{\leftarrow d_1} $,  every special vertex $d_1$ add  $\ell_{\rightarrow d_1}\ell_{\leftarrow d_1}$ to the Frobenius dimension.
\end{obs}

\begin{prop}\label{agregar un Y-tail a un lineal}
Assume that we know $\D$ for the algebra $\A$ given by  the quiver $Q$: $$\SelectTips{eu}{10}\xymatrix@R=.8pc@C=.5pc{ \ar@{.}[rr]& & \scriptstyle{i-1} \ar[rr]^{\a_{i-1}} & & \scriptstyle{i} \ar[rr]^{\a_i} && \scriptstyle{i+1}  \ar[rr]^{\a_{i+1}}  &&   \scriptstyle{i+2} \ar@{.}[rr]  & & }$$
 with $i,i+1$  vertices of valency 2.  Let  $Q'$  be the quiver obtained by adding to $Q$ a tail  through the arrow $\a_i$. Let $\D '$ be a coproduct for the algebra $\A'$ associated to $Q'$.  Then  $\D'(e_v)=0$ for all $v\in Q'_0\setminus Q_0$.  Moreover,  $\Fr\A=\Fr{\A'}$.
\end{prop}

\begin{proof}
If we prove that $\D'(\c)=\D'(\b)=0$, the first assertion follows from lemma \ref{delta en el tail da cero}. We have $\D'(e_{i+1})= \displaystyle{\sum_{\sigma,\eta} G_{\sigma,\eta} \sigma \otimes \eta}$ with $s(\sigma)=t(\eta)=i+1$. Since the vertex $i+1$ has valency two in $Q$, we have  $\eta=\eta'\a_i$ with $t(\eta ')=i$ or $\eta=e_{i+1}$. We claim that every $G_{\sigma,e_{i+1}}$ has to be zero. Indeed, $\D'(\a_i)=\displaystyle{\sum_{\sigma} G_{\sigma,e_{i+1}} \a_i\sigma \otimes e_{i+1}}+\displaystyle{\sum_{\sigma,\eta '} G_{\sigma,\eta'} \a_i\sigma \otimes \eta'\a_i}$ \hspace{2mm}  but \hspace{2mm} $\D'(\a_i)=\D'(e_i \a_i)=\displaystyle{\sum_{\mu,\varepsilon} G_{\mu,\varepsilon} \mu \otimes \varepsilon}\a_i$   with   $s(\mu)=t(\varepsilon)=i$,    which do not have any term $\a_i\sigma \otimes e_{i+1}$. Thus, $\D'(e_{i+1})= \displaystyle{\sum_{\sigma,\eta'} G_{\sigma,\eta'} \sigma \otimes \eta '\a_i}$ and, in consequence, $\D'(\c)=0$. Similarly we prove that $\D'(\b)=0$.\\

Now, we want to show that $\D(e_v)=\D'(e_v)$ for all $v\in Q_0$. Clearly, since $\b\a_i=\a_i\c=0$,  $\D(e_v)=\D'(e_v)$ for all $v\in Q_0 \setminus \{i,i+1\}$. For $v=i$ observe that $\D'(e_i)=\D(e_i)+ \displaystyle{\sum_{\sigma,\eta} G_{\sigma,\eta} \sigma \otimes \eta\b}$ with $s(\sigma)=i$ and $t(\eta)=s(\b)$.  Then,

\begin{eqnarray*}
\D(\a_{i-1}) & = &  \D(e_{i-1})(1\otimes \a_{i-1}) \hspace{2mm}   =  \hspace{2mm}  \D'(e_{i-1})(1\otimes \a_{i-1})    \\
& =&  \D'(\a_{i-1}) \hspace{2mm}  =  \hspace{2mm}  (\a_{i-1}\otimes 1)\D'(e_i) \\
& = & \D(\a_{i-1}) + \displaystyle{\sum_{\sigma,\eta} G_{\sigma,\eta} \a_{i-1}\sigma \otimes \eta\b}
\end{eqnarray*}

which implies that every scalar $ G_{\sigma,\eta}=0$ and consequently $\D'(e_i)=\D(e_i)$. Similar arguments applies for $v=i+1$. Finally,
$\Fr\A=\Fr{\A'}$.
\end{proof}

Recall that a \emph{gluing} of two quivers with relations $Q$ and $Q'$ with respect to a vertex  $v$ in $Q$ and $v'$ in $Q'$ is obtained
by taking the disjoint union $Q \sqcup Q'$ and identifying the vertices $v$ and $v'$. The relations are induced from those of $Q$ and $Q'$ (i.e. there are no additional relations).

The following lemma show what happens if we glue a saturated cycle at the ' final'  vertex of a tail.

\begin{lema}\label{lineal-freecycle}
Let $\A$ be the algebra given by the bound quiver

$$\SelectTips{eu}{10}\xymatrix@R=.5pc@C=.5pc{&&&&&&&&&&&&&&& \\
\scriptstyle{a+1} \ar@{.}[u] \ar[rd]^{\c}  & &&&&&&&& & \scriptstyle{m+1} \ar[rd]^{\varepsilon}  & &&& \\
 & \scriptstyle{1} \ar[dl]^{\b} \ar@{-}[rr]^{\b_2}& & \scriptstyle{2}  \ar@{-}[rr]^{\b_3} &&    \ar@{.}[rr]& &   \ar[rr]  && \scriptstyle{m} \ar[ru]^{\b_{m+1}}  &&  \scriptstyle{d} \ar[ll]^{\mu} \\
\scriptstyle{a}\ar[uu]^{\a}  \ar@{.}[d] &&&&&&&&&&&&&&&&&&\\
&&&&&&&&&&&&&&&&&&&}$$

Assume that $\D(\b)=\D(\c)=0$, then  $\D(e_i)=0$ for all $ i=1,\cdots, m$ ;  $\D(e_{m+1})  =  \displaystyle{\sum_{\eta / t(\eta)=m+1} G_{\eta} \varepsilon \otimes \eta  } $ and  \
$\D(e_d)=  J \mu\otimes \varepsilon+ G_{e_{m+1}} e_d\otimes\varepsilon$.

\end{lema}

\begin{proof}
As in lemma \ref{delta en el tail da cero}  $\D(e_{i})=0$ for all $ i=1,\cdots, m$. For $e_{m+1}$,
$$\D(e_{m+1})=\displaystyle{\sum_{t(\eta)=m+1} G_{\eta} \varepsilon \otimes \eta }+ \displaystyle{\sum_{t(\eta)=m+1} G'_{\eta} e_{m+1} \otimes \eta}.$$
 Comparing with  $\D(e_{m})=0$ using the arrow $\b_{m+1}$ and the fact that $\b_{t+1}\varepsilon=0$, we deduce that every scalar $G'_{\eta}$ has to be zero. Then  $\D(e_{m+1})=\displaystyle{\sum_{t(\eta)=m+1} G_{\eta} \varepsilon \otimes \eta }$.\

A general expression for  $\D(e_d)= J\mu\otimes \varepsilon +K \mu\otimes e_d +L e_d\otimes\varepsilon$. Comparing with $\D(e_{d_{m+1}})$ using the arrow $\varepsilon$ we obtain $L=G_{e_{m+1}}$ and comparing with $\D(e_{d_{m}})$ using the arrow $\mu$ we get $K=0$. Finally  $\D(e_d)= J \mu\otimes \varepsilon+ G_{e_{m+1}} e_d\otimes\varepsilon$.
\end{proof}

\begin{coro}\label{tail con free cycle al final}
Let $Q$ be the  following quiver with $\small{a,a+1}$  vertices of valency 2. $$\SelectTips{eu}{10}\xymatrix@R=.8pc@C=.5pc{ \ar@{.}[rr]& & \scriptstyle{a-1} \ar[rr] & & \scriptstyle{a} \ar[rr]^{\a_a} && \scriptstyle{a+1}  \ar[rr]  &&   \scriptstyle{a+2} \ar@{.}[rr]  & & }$$

Let  $Q'$  be the quiver obtained by adding to $Q$ a tail $\SelectTips{eu}{10}\xymatrix@R=.5pc@C=.5pc{\scriptstyle{1} \ar@{-}[rr]& & \scriptstyle{2}    \ar@{.}[rr]& &   \ar@{-}[rr]  && \scriptstyle{m} }$ through the arrow $\a_a$ and gluing  a saturated cycle
$\SelectTips{eu}{10}\xymatrix@R=.4pc@C=.2pc{&  \scriptstyle{m+1}\ar[dr]  &     \\ \scriptstyle{m} \ar[ru]& & \scriptstyle{d}    \ar[ll] }$ at the vertex  $m$ of the tail. \

Name $\A$ the algebra  given by the quiver $Q$ and  $\A'$  the algebra associated to $Q'$. Then,  $$\Fr{\A'}=\Fr\A+ \ell_{\rightarrow m+1}+2.$$
\end{coro}

\begin{proof}
According to proposition   \ref{agregar un Y-tail a un lineal} and lemma \ref{lineal-freecycle} (or its dual) it is enough to observe  that  the number of different $G_{\eta}$ at lemma \ref{lineal-freecycle} is $|(\K Q')e_{m+1}|=\ell_{\rightarrow m+1}+1$  and $\D(e_d)$ add one more for the scalar $J$.\

Observe that $\ell_{\rightarrow m+1}+2=\ell_{\rightarrow m+1}\ell_{\leftarrow m+1}+\ell_{\rightarrow d}\ell_{\leftarrow d}+1$.
\end{proof}

Now we are ready to state and prove the main result of this section. \

\begin{teo}
If $\A$ is a  cluster tilted algebra of type $\mathbb{A}$, then $\A$  has  finite Frobenius dimension. Moreover,
$$\Fr\A= \sharp \mathcal{B} + \displaystyle{\sum_{b\in \mathcal{V_S}} \ell_{\rightarrow b}.\ell_{\leftarrow b} } $$

\end{teo}


\begin{proof}
We first claim   that for every basis path we have a linear structure. Let  $\mathcal{P} : 1 \rightsquigarrow n$ be a basis path. If $\mathcal{P}$ is of the first type  \cite[Lemma 1]{AGL} implies that $\D_{\mid \mathcal{P}}$ is the linear structure. If $\mathcal{P}$ is of the second type (or third type)    lemma \ref{lineal-freecycle} (or lemma \ref{caso-particular-freecycle-lineal-freecycle} respectively ) shows that  the coproduct $\D$ restricted to $\mathcal{P}$ is the linear structure.   Observe that if $\A$ does not have basis paths, according to lemmas \ref{tail-con-vertice-de-valencia-dos} and  \ref{los-tres-tails-salientes} $\D$ restricted to $Q_0\setminus \mathcal{V_S}$ is zero.\


Now assume that we have two basis paths $\mathcal{P}$ and $\mathcal{P}'$ that share a vertex $b$, then we have two cases:

\begin{enumerate}
  \item $b=t(\mathcal{P})=s(\mathcal{P}')$ is an special vertex:  using corollary \ref{lineal-freecycle-lineal} we conclude that we have two independent linear structures, one for each basis path. Moreover, from the same result we obtain that the special vertex  $b$ add $\ell_{\rightarrow b}.\ell_{\leftarrow b}$ to the $\Fr\A$.
  \item $b$ is a vertex of valency four: we can assume that $\mathcal{P}: \SelectTips{eu}{10}\xymatrix@R=.8pc@C=.5pc{ \cdots \ar[rr]& & \scriptstyle{c_1} \ar[rr] & & \scriptstyle{b} \ar[rr] && \scriptstyle{c_3}  \ar[rr]  &&  \cdots & }$ is of any type and let $\D_{\mathcal{P}}$ be the coproduct restricted to the algebra $\A'$ associated to the quiver of $\mathcal{P}$ (i.e. a linear structure). Then we have the following  three different situations depending on the type of $\mathcal{P}'$.

\begin{itemize}

  \item [a)] $\mathcal{P}': f\rightsquigarrow d_2 \rightarrow b \rightarrow d_1 \rightsquigarrow g$ of the first type:
$\SelectTips{eu}{10}\xymatrix@R=.8pc@C=.5pc{ \cdots \ar[rr]& & \scriptstyle{c_1} \ar[rr] & & \scriptstyle{b} \ar[rr] \ar[dl]_{\b_2}  && \scriptstyle{c_3} \ar[dl]^{\b_3}  \ar[rr]  && \cdots & \\ & & & \scriptstyle{d_1} \ar@{~>}[d]_{\a'}\ar[ul]^{\d_2} & & \scriptstyle{d_2} \ar[ul]^{\d_3} & &&&  \\
&&& \scriptstyle{g}  && \scriptstyle{f} \ar@{~>}[u]_{\a} &&&&}$

Then,  a simple computation gives

$$\begin{array}{l}
   \D(e_{f})= A \a\d_3\b_2\a'\otimes e_f \\
   \hspace{.5cm} \vdots \\
   \D(e_{d_2})= A \d_3\b_2\a'\otimes \a\\
   \D(e_{b})= A \b_2\a'\otimes \a\d_3 + \D_{\mathcal{P}}(e_{b})\\
   \D(e_{d_1})= A \a'\otimes \a\d_3\b_2\\
   \hspace{.5cm} \vdots \\
   \D(e_{g})= A e_g \otimes \a\d_3\b_2\a'
\end{array}$$

  \item [b)] $\mathcal{P}': f\rightsquigarrow d_2 \rightarrow b \rightarrow d_1 $ of the second type: $\SelectTips{eu}{10}\xymatrix@R=.8pc@C=.5pc{ \cdots \ar[rr]& & \scriptstyle{c_1} \ar[rr] & & \scriptstyle{b} \ar[rr] \ar[dl]_{\b_2}  && \scriptstyle{c_3} \ar[dl]^{\b_3}  \ar[rr]  && \cdots & \\ & & & \scriptstyle{d_1} \ar[ul]^{\d_2} & & \scriptstyle{d_2} \ar[ul]^{\d_3} & &&&  \\
&&&  && \scriptstyle{f} \ar@{~>}[u]_{\a} &&&&} $

Assume that $\a:f \rightsquigarrow d_2=\a_1\cdots \a_n$. Then,
$$\begin{array}{l}
   \D(e_{f})= A \a\d_3\b_2\otimes e_f\\
   \hspace{.5cm} \vdots \\
   \D(e_{d_2})= A \d_3\b_2\otimes \a \\
   \D(e_{b})= A \b_2\otimes \a\d_3 + \D_{\mathcal{P}}(e_{b})\\
   \D(e_{d_1})= A e_{d_1}\otimes \a\d_3\b_2 + B_1 \d_2\otimes \b_2 + B_2 \d_2\otimes \d_3\b_2 + \displaystyle{\sum_{i=1}^n B_{3+n-i} \d_2\otimes  \a_i\cdots \a_n\d_3\b_2}
\end{array}$$

  \item [c)] $\mathcal{P}': d_2 \rightarrow b \rightarrow d_1 $ of the third type:  $\SelectTips{eu}{10}\xymatrix@R=.8pc@C=.5pc{  \cdots \ar[rr]& & \scriptstyle{c_1} \ar[rr] & & \scriptstyle{b} \ar[rr] \ar[dl]_{\b_2}  && \scriptstyle{c_3} \ar[dl]^{\b_3}  \ar[rr]  && \cdots & \\ & & & \scriptstyle{d_1} \ar[ul]^{\d_2} & & \scriptstyle{d_2} \ar[ul]^{\d_3} & &&& }$

A simple computation gives
$$\begin{array}{l}
   \D(e_{b})= A \b_2\otimes \d_3 + \D_{\mathcal{P}}(e_{b})\\
  \D(e_{d_1})= A e_{d_1}\otimes \d_3 \b_2 + B \d_2\otimes \b_2 + C \d_2\otimes \d_3 \b_2 \\
  \D(e_{d_2})= A \d_3 \b_2 \otimes e_{d_2} + D \d_3\otimes \b_3 + E \d_3 \b_2\otimes \b_3
\end{array}$$
\end{itemize}

\end{enumerate}

Then we conclude that $\mathcal{P}'$ add an  independent linear structure. Also, we obtain that in the cases where $d_i$ is an  special vertex it  add $\ell_{\rightarrow d_i}.\ell_{\leftarrow d_i}$ to the $\Fr\A$. In general, according to corollary \ref{tail con free cycle al final}, every special vertex $d$ add $\ell_{\rightarrow d}.\ell_{\leftarrow d}$ to the $\Fr\A$.\

It remaind to proof that $\D(e_v)=0$ for all  non special vertex  $v$ not belonging to any basis path. It follows from lemma \ref{los-tres-tails-salientes}, proposition \ref{agregar un Y-tail a un lineal} and corollary \ref{tail con free cycle al final}.

\end{proof}

\begin{coro}
Let $\A$ be a cluster tilted algebra of type $\mathbb{A}$. Then $\Fr\A\geq 1$ if and only if $\sharp \mathcal{V_S}  + \sharp \mathcal{B}\neq 0$
\end{coro}

%
%
%

%

\begin{ejem}
 The algebra  $\A$  given by the following bound quiver has  $\Fr\A=0$.

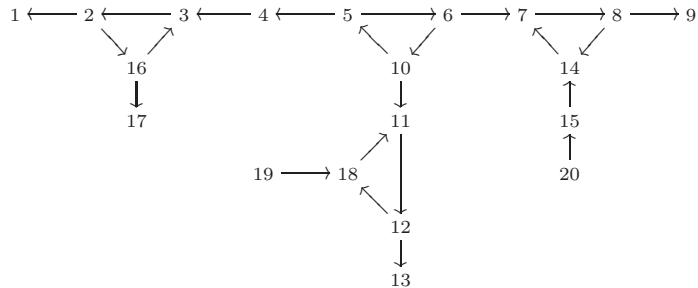
$$\SelectTips{eu}{10}\xymatrix@R=.8pc@C=.5pc{ \scriptstyle{1} & & \scriptstyle{2} \ar[dr] \ar[ll]_{}  &&   \scriptstyle{3} \ar[ll]_{} & & \ar[ll] \scriptstyle{4}  && \scriptstyle{5}  \ar[ll]_{} \ar[rr]^{}  &&  \scriptstyle{6} \ar[dl]^{} \ar[rr]^{} & & \scriptstyle{7}  \ar[rr] && \scriptstyle{8}  \ar[dl] \ar[rr]^{}  && \scriptstyle{9} \\
&&& \scriptstyle{16} \ar[d]\ar[ru] &&&&& & \scriptstyle{10} \ar[lu]^{} \ar[d]^{} & &&& \scriptstyle{14}  \ar[ul]  &&&& \\
&&& \scriptstyle{17} &&&&& & \scriptstyle{11} \ar[dd] & &&& \scriptstyle{15} \ar[u] &&& \\
&&&&&&\scriptstyle{19}\ar[rr] && \scriptstyle{18} \ar[ur] & & &&& \scriptstyle{20} \ar[u] &\\
&&&&&&&& & \scriptstyle{12}  \ar[ul]\ar[d]^{} & &&&&&& \\
&&&&&&&& & \scriptstyle{13}  & &&&&&&}$$

%

\end{ejem}

\subsection{Further consecuences}

 Given  a quiver $Q$ without loops and
$2$-cycles and a vertex $k$, we denote by $\mu_k(Q)$ the Fomin-Zelevinsky quiver
mutation \cite{FZ02} of $Q$ at $k$. Two quivers are called
\emph{mutation equivalent} if one can be reached from the other by a
finite sequence of quiver mutations.
In particular, two quivers are called \emph{sink/source equivalent} if one can be obtained
from the other by performing mutations only at vertices which are sinks
or sources.

\begin{ejem}This example show that the Frobenius dimension is not invariant under sink/source mutations.

$$\begin{array}{ccc}
  \SelectTips{eu}{10}\xymatrix@R=.8pc@C=.6pc{ \scriptstyle{1} \ar[rr]& & \scriptstyle{2}   \ar[dl] \ar[rr] &&   \scriptstyle{3} & &  \scriptstyle{4} \ar[ll] \\
 & \scriptstyle{5} \ar[lu]  & &&&&&& } & \stackrel{\mu_4}{\longmapsto} \hspace{.5cm} & \SelectTips{eu}{10}\xymatrix@R=.8pc@C=.6pc{ \scriptstyle{1} \ar[rr]& & \scriptstyle{2}   \ar[dl] \ar[rr] &&   \scriptstyle{3} \ar[rr] & &  \scriptstyle{4} \\
 & \scriptstyle{5} \ar[lu]  & &&&&&& } \\
&&\\
\hspace{-1cm} \Fr\A=4 && \Fr {\mu_4(\A)}=6
\end{array}$$

\end{ejem}

\begin{ejem} Moreover the fact of have non-zero Frobenius dimension neither is an invariant under mutation as we can see at this example where the algebras $\A$ and $\A'$ are mutation equivalent.

$$\begin{array}{ccc}
  \SelectTips{eu}{10}\xymatrix@R=.8pc@C=.6pc{ \scriptstyle{1}  & & \scriptstyle{2} \ar[rr] \ar[ll] & &  \scriptstyle{3}  \ar[dl] \ar[rr] &&   \scriptstyle{4}   &&  & \\ & & &  \scriptstyle{5} \ar[d] \ar[ul] & &  & &&&  \\
&&&  \scriptstyle{6}   &&   &&&&}  $$
  &  \hspace{.5cm} & \SelectTips{eu}{10}\xymatrix@R=.8pc@C=.6pc{ \scriptstyle{1} \ar[rr]& & \scriptstyle{2}   \ar[dl] \ar[rr] &&   \scriptstyle{3} \ar[rr] & &  \scriptstyle{4} \ar[rr] && \scriptstyle{6} \\
 & \scriptstyle{5} \ar[lu]  & &&&&&& } \\
&&\\
\hspace{-1cm}\Fr\A=0 && \Fr {\A'}=7
\end{array}$$

\end{ejem}

In particular, we conclude that the Frobenius dimension is not a derived invariant. \\

From the examples we can observe that the exact number of the Frobenius dimension is not really interesting but we are still interested in knowing whatever or not a cluster tilted  algebra has non-zero Frobenius dimension and equivalently admit at least one open Frobenius structure.\\

\section{Open Frobenius Cluster-Tilted Algebras of type $\mathbb{D}$}



Following \cite{V} we will describe the quivers of cluster-tilted algebras of type $\mathbb{D}$.
Let $\MM^A_k$ be the mutation class of $\mathbb{A}_k$. The union of all $\MM^A_k$ for all $k$ will be denoted by $\MM^A$.
For a quiver $\Gamma$ in $\MM^A$, we will say that a vertex $v$ is a \emph{connecting vertex} if $v$ has valency 1 or if $v$ has valency
2 and  $v$ belongs to a 3-cycle in $\Gamma$.
Let $Q$  be a quiver with $n$ vertices having a full subquiver $\widehat{Q}$ of one of the following four types:\\

\begin{center}
\begin{tabular}{|c|c|c|c|}
  \hline
    Type I    & Type II   & Type III  &  Type IV \\ \hline
    &&& \\
    \SelectTips{eu}{10}\xymatrix{ a \ar@{-}[dr] & \\ & c  \\ b \ar@{-}[ur] &   } & \SelectTips{eu}{10}\xymatrix{ & b \ar[dl]_{\d} & \\ d \ar[rr]^{\mu} && c \ar[dl]^{\beta} \ar[ul]_{\a} \\ & a \ar[lu]^{\c} &}
    &  \SelectTips{eu}{10}\xymatrix{ & b \ar[dl]_{\b} & \\ d \ar[rd]_{\c} && c  \ar[ul]_{\a} \\ & a \ar[ru]_{\d} &}  &  \SelectTips{eu}{10}\xymatrix@R=.4pc@C=.3pc{  &&& c_{\a_2} \ar[ddl]_{\c_2} &&&&&  \\&&&&&&&& \\ c_{\a_1} \ar[dd]_{\c_1}& & . \ar[rr]_{\a_2} \ar[ll]_{\b_1} && . \ar[uul]_{\b_2}\ar[ddrr] &&&&  \\&&&&&&&& \\. \ar[rruu]_{\a_1}&&&& &&.\ar[dd]_{\a_i}&&  \\&&&&&&&& c_{\a_i} \ar[ull]_{\c_i}  \\. \ar[uu]_{\alpha_k} \ar@{.}@/_/[rrrrrr]&&&&&& . \ar[rru]_{\b_i} &&} \\
    &&&\\
    &&& \\\hline
\end{tabular}
\end{center}

\medskip

For $\widehat{Q}$ of type I, II or III,  let $Q'$ be $(Q\setminus \widehat{Q})\cup \{c\}$ and  $Q''$ be $(Q\setminus \widehat{Q})\cup \{d\}$. For $\widehat{Q}$ of type IV the cycle $\alpha_1\alpha_2 \cdots \alpha_k $  is a directed $k$-cycle (called the \emph{central cycle}), with $k\geq 3$. For each arrow $\alpha :a\to b$ in the central cycle, there may (and may not) be a vertex $c_{\alpha}$ which is not on the central cycle, such that there is an oriented 3-cycle $a\overset{\alpha}{\to} b\to c_{\alpha}\to a$. Such a 3-cycle will be called a \emph{spike}. For each spike let $Q^{\alpha}$ be the quiver $(Q\backslash \widehat{Q}) \cup \{c_{\alpha} \}$.\\

 We say that $Q$ is of:

 \begin{itemize}
   \item   \underline{Type I}  if  $Q$ has a full subquiver $\widehat{Q}$ of type I, the vertices $a$ and $b$  have valency one and both are a sink (or a source), $Q'$ is in $\MM^A_{n-2}$ and $c$ is a connecting vertex for $Q'$.
   \item  \underline{Type II} (or \underline{Type III} ) if  $Q$ has a full subquiver $\widehat{Q}$ of type II (or type III, respectively),   the vertices $a$ and $b$ have valency 2, $Q'$ and $Q''$ are both in $\MM^A$ and have $c$  and $d$ as connecting vertices respectively.
   \item \underline{Type IV} if  $Q$ has a full subquiver $\widehat{Q}$ of type IV, $Q^{\alpha}$ is in $\MM^A$ and has the vertex  $c_{\alpha}$ as a connecting vertex.
 \end{itemize}

Observe that in types II, III and IV, the subquivers $Q', Q'', Q^{\alpha}$ can be in the set $\MM^A_1$, i.e. they can have only one vertex.
We define $\MM^D_i$ ($i\in \{$ I,II,III,IV $ \})$ to be the set of quivers $Q$   belonging to the  type $i$ described above. \\

\subsection{Cluster-tilted algebras of type $\mathbb{D}$  and sub-type I }

According to \cite{BMR},  $A=\K Q/ I $ is a cluster-tilted algebra of type $\mathbb{D}$  and sub-type I if and only if   $Q\in \MM^D_ {\rm I}$ and every 3-cycle is \textit{saturated}. \\

Given a vertex $c\in Q_0$ let $\mathcal{B}_{c}$  be the set of all basis paths  $\mathcal{P} : v_1 \rightsquigarrow v_t$ such that $v_i=c$ for some $1\leq i\leq t$.

\begin{prop}
Let $A=\K Q/ I $ be  a cluster-tilted algebra of type $\mathbb{D}$  and sub-type I, then $$\Fr\A= \sharp (\mathcal{B}\setminus \mathcal{B}_{c})  + \displaystyle{\sum_{b\in \mathcal{V_S}} \ell_{\rightarrow b}.\ell_{\leftarrow b} } $$

\end{prop}

\begin{proof}

It is enough to observe that according to \cite[Lemma 2]{AGL} $\D(e_c)=0$. Then $\D\mid_{\mathcal{P}}\equiv 0$, for every basis path $\mathcal{P}\in \mathcal{B}_{c}$. Moreover,  $\D(e_a)=\D(e_b)=0$.
\end{proof}

\subsection{Cluster-tilted algebras of type $\mathbb{D}$  and sub-type II }

$A=\K Q/ I $ is a cluster-tilted algebra of type $\mathbb{D}$  and sub-type II if and only if   $Q\in \MM^D_ {\rm II}$, $\a\d=\b\c$, $\mu \a=\mu\b=\d\mu=\c\mu=0$ and every 3-cycle in $Q'\cup Q''$ is \textit{saturated}  \cite{BMR}. \\

We adapt the definition \ref{def de camino base y vertice especial} to this case.
 We say that
\begin{itemize}
  \item [a)] a vertex $v\in Q_0$ is an \emph{special vertex} if $v$ has valency two and there are $\alpha, \beta \in Q_1$ with $s(\b)=t(\a)=v$ and $\a\b\in I$ or $v\in \{c,d\}$ and has valency three. Let denote by $\mathcal{V_S}^*$ the new set of  special vertices.
  \item [b)] an  \emph{extended basis path }  is a path $\mathcal{P} : v \rightsquigarrow v'$ such that $v,v'$ are or an special vertex  or a source or a sink of valency one. We denote by   $\mathcal{B}^*$ the set of all extended basis paths.
\end{itemize}


Given a vertex $z$ we will denote by $\sharp_{\leftarrow z}=dim_{\K}( e_z(\K Q/ I)) - 1$  (or $\sharp _{\rightarrow z}= dim_{\K} ((\K Q/ I)e_z) -1 $ )  the number of non trivial paths with source ( or target, respectively ) $z$.

\begin{prop}
Let $A=\K Q/ I $ be  a cluster-tilted algebra of type $\mathbb{D}$  and sub-type II, then $$\Fr\A= \sharp \mathcal{B}^*  + \displaystyle{\sum_{b\in \mathcal{V_S}^*} \sharp_{\rightarrow b}.\sharp_{\leftarrow b} } $$

\end{prop}

\begin{proof}
First observe that in this case  we possible  add the vertices $c$ and $d$ to the set of special vertices that we already had for types $\mathbb{A}$ and  $\mathbb{D}$ sub-type I. In this way we possible have new basis paths ( called extended basis paths ) that contribute also with linear structures.  Then, as for the  $\mathbb{A}$-case,  for every extended basis path we have a linear structure.\\
When we compute the Frobenius dimension  and we look at an special vertex $z$, we were interested in the  number of non trivial paths with source (or target ) $z$, in other words, the numbers $dim_{\K}( e_z(\K Q/ I)) - 1$ and $dim_{\K} ((\K Q/ I)e_z) -1 $.  Cluster tilted algebras of types $\mathbb{A}$  are gentle which implies that those numbers coincide with the length of the largest path with source (or target, respectively )  $z$.
Finally, an adaptation of the formula for the  $\mathbb{A}$-case to this one gives  our claim.

\end{proof}

We illustrate the case with the following example.

\begin{ejem}

Let $\A$ be the algebra  given by the following  quiver

$$\SelectTips{eu}{10}
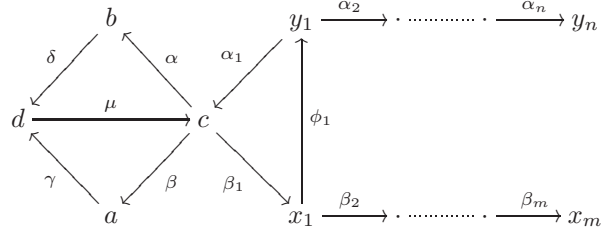
\xymatrix{ & b \ar[dl]_{\d} &    & y_1 \ar[dl]_{\a_1} \ar[r]^{\a_2} & \cdot \ar@{.}[r]  & \cdot  \ar[r]^{\a_n}& y_n  \\ d \ar[rr]^{\mu} && c \ar[dl]^{\beta} \ar[ul]_{\a} \ar[dr]_{\b_1}  &&&&
\\ & a \ar[lu]^{\c} & & x_1 \ar[uu]_{\phi_1}  \ar[r]^{\b_2} & \cdot \ar@{.}[r]  & \cdot  \ar[r]^{\b_m}& x_m}$$
\end{ejem}

bounded by the relations $\a\d=\b\c$, $\mu \a=\mu\b=\d\mu=\c\mu=\phi_1\a_1=\a_1\b_1=\b_1\phi_1=0$.\\

A simple computation gives

\begin{eqnarray*}
\D(e_b) &=& 0 \\
\D(e_a) &=& 0 \\
\D(e_{y_i}) &=& 0 \text{ \ \ ; $1 \leq i \leq n$}\\
\D(e_{x_j}) &=& A \b_{j+1}\cdots \b_{m}\otimes \mu \b_1 \cdots \b_j \\
\D(e_c) &=& A \b_{1}\cdots \b_{m}\otimes \mu\\
\D(e_d) &=& A \mu\b_{1}\cdots \b_{m}\otimes e_d+ B \mu\otimes \delta  + C \mu\otimes \c +D \mu\otimes \a\delta + E \mu\otimes \a_1\a\delta + \displaystyle{\sum_{i=1}^m B_{i}  \mu\b_{1}\cdots \b_{i}\otimes \delta } \\
 & & + \displaystyle{\sum_{i=1}^m C_{i}  \mu\b_{1}\cdots \b_{i}\otimes \c }+ \displaystyle{\sum_{i=1}^m D_{i}  \mu\b_{1}\cdots \b_{i}\otimes \a\delta }+ \displaystyle{\sum_{i=1}^m E_{i}  \mu\b_{1}\cdots \b_{i}\otimes \a_1\a\delta  }
\end{eqnarray*}

Then $\Fr\A=1+4(m+1)$. \\

For the last two sub-types we are going to find a lower bound for the Frobenius dimension.

\subsection{Cluster-tilted algebras of type $\mathbb{D}$  and sub-type III }

$A=\K Q/ I $ is a cluster-tilted algebra of type $\mathbb{D}$  and sub-type III if and only if   $Q\in \MM^D_ {\rm III}$, $\a\b\c=\b\c\d=\c\d\a=\d\a\b=0$ and every 3-cycle in $Q'\cup Q''$ is \textit{saturated}  \cite{BMR}. \\

\begin{prop}
Let $A=\K Q/ I $ be  a cluster-tilted algebra of type $\mathbb{D}$  and sub-type III, then $$\Fr\A \geq 2 $$

\end{prop}

\begin{proof}
It is enough to see that $\D(e_a) = A \delta\a\otimes \b\c +  \sum_{\rho,\eta} G_{\rho,\eta} \rho \otimes \eta$  with $s(\rho)=t(\eta)=a$ and $\rho \neq \delta\a $ , $\eta\neq \b\c$. Since $\b\c\d=\c\d\a=0$ we are sure that $A\neq 0$. \\
 In the same way, we can affirm that $\D(e_b) = B \b\c\otimes \delta\a +  \sum_{\eta,\rho} G_{\eta,\rho} \eta\otimes \rho$  with $t(\rho)=s(\eta)=b$ and $\rho \neq \delta\a $ , $\eta\neq \b\c$. Since $\a\b\c=\d\a\b=0$ we are sure that $B\neq 0$.
\end{proof}

The following example shows an algebra with Frobenius dimension exactly two.

\begin{ejem}

Let $\A$ be the algebra  given by the following  quiver

$$\SelectTips{eu}{10}
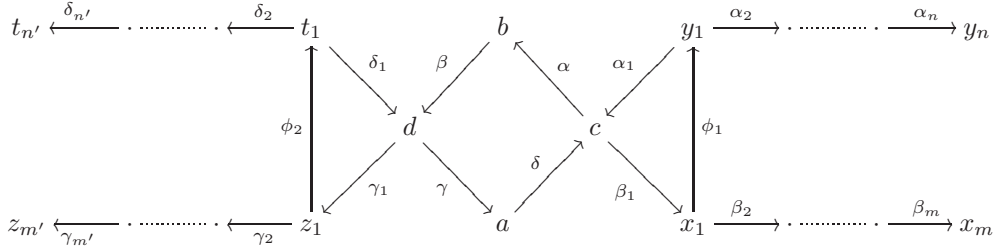
\xymatrix{t_{n'} & \cdot  \ar[l]_{\delta_{n'}} \ar@{.}[r]  & \cdot & t_1 \ar[l]_{\delta_2} \ar [dr]^{\delta_1}  &    & b \ar[dl]_{\b} &    & y_1 \ar[dl]_{\a_1} \ar[r]^{\a_2} & \cdot \ar@{.}[r]  & \cdot  \ar[r]^{\a_n}& y_n  \\ &  &  &   & d \ar[dr]_{\c}  \ar[dl]^{\c_1}&& c \ar[ul]_{\a} \ar[dr]_{\b_1}  &&&&
\\  z_{m'} & \cdot  \ar[l]^{\c_{m'}} \ar@{.}[r]  & \cdot & z_1 \ar[l]^{\c_2} \ar [uu]^{\phi_2}  &    &  a \ar[ur]^{\delta} & & x_1 \ar[uu]_{\phi_1}  \ar[r]^{\b_2} & \cdot \ar@{.}[r]  & \cdot  \ar[r]^{\b_m}& x_m}$$
\end{ejem}

bounded by the relations $\a\b\c=\b\c\d =\c\d\a=\d\a\b=0$, $\delta_1\gamma_1=\gamma_1\phi_2=\phi_2\delta_1=0$ ,$\phi_1\a_1=\a_1\b_1=\b_1\phi_1=0$.\\

A simple computation gives

\begin{eqnarray*}
\D(e_d) &=& \D(e_c) = 0 \\
\D(e_{y_i}) &=& 0 \text{ \ \ ; $1 \leq i \leq n$}\\
\D(e_{x_i}) &=&  0 \text{ \ \ ; $1 \leq i \leq m$}\\
\D(e_{t_i}) &=& 0 \text{ \ \ ; $1 \leq i \leq n'$}\\
\D(e_{z_i}) &=& 0 \text{ \ \ ; $1 \leq i \leq m'$}\\
\D(e_b) &=& B \b\c\otimes \delta\a \\
\D(e_a) &=& A \delta\a\otimes \b\c
\end{eqnarray*}

Then $\Fr\A=2$. \\

\subsection{Cluster-tilted algebras of type $\mathbb{D}$  and sub-type IV }

$A=\K Q/ I $ is a cluster-tilted algebra of type $\mathbb{D}$  and sub-type IV if and only if   $Q\in \MM^D_ {\rm IV}$,  every 3-cycle in $\cup_i Q^{\a_i}$ is \textit{saturated} and for each arrow $\a_i$ in the central cycle $\a_{i+1}\cdots \a_k\a_1\cdots \a_{i-1}=\b_i\c_i$ if there exits $c_{\a_i}$ or $\a_{i+1}\cdots \a_k\a_1\cdots \a_{i-1}=0$ otherwise. Indices are read modulo $k$.  \cite{BMR}. \\

\begin{prop}
Let $A=\K Q/ I $ be  a cluster-tilted algebra of type $\mathbb{D}$  and sub-type IV with a central cycle of length $k$, then $$\Fr\A \geq k $$

\end{prop}

\begin{proof}

Let us start with the case where there are not spikes.  Then, for each arrow $\a_i$  in the central cycle we have $\a_{i+1}\cdots \a_k\a_1\cdots \a_{i-1}=0$.  We affirm that $\D(e_2)$ has the term $A\a_2\cdots\a_{k-1}\otimes\a_4\cdots\a_k\a_1$ with $A\neq 0$. Since $\a_{1}\cdots \a_{k-1}=0$ and $\a_{4}\cdots \a_k\a_1\a_{2}=0$, the scalar $A$ does not appear in the expression of $\D(\a_1)$ and $\D(\a_2)$ and, in consequence, does not appear in $\D(e_i)$ with $i \neq 1$. The same reasoning applied to any vertex $i$ gives $\Fr\A\geq k$.\

We now turn to the case where there are  spikes.  We can assume that there exists an spike $c_{\a_1}$. Then, $\a_2\cdots\a_{k}=\b_1\c_1$ and $\a_1\b_1=\c_1\a_1=0$. We affirm that $\D(e_2)$ has the term $A\a_2\cdots\a_{k}\otimes\a_1$ with $A\neq 0$. Since  $\a_1\b_1=0$ and $\a_2\cdots\a_{k}=\b_1\c_1$, the scalar $A$ does not appear in the expression of $\D(\b_1)$ and $\D(\a_1)$. On the other hand the term $A\a_2\cdots \a_k\otimes \a_1\a_2$ has to appear in the expression of $\D(\a_2)$ and consequently $A \a_3\cdots \a_k\otimes\a_1\a_2\a_3$ appears in the expression of $\D(e_3)$. If there are not others spikes, we continue with this reasoning showing that  $A \a_{i}\cdots \a_k\otimes\a_1\cdots\a_{i-1}$ appears in the expression of $\D(e_i)$ for $i< k$ and does not appear for $i=k$ because $\a_1\cdots\a_{k-1}=0$. This establishes $A\neq 0$. We now suppose that there exits an spike $c_{\a_i}$. It follows, by the same method  as before, that $A \a_{i}\cdots \a_k\otimes\a_1\cdots\a_{i-1}$ appears in the expression of $\D(e_i)$. Since $\c_i\a_i=0$ the scalar $A$ does not appear in the expression of $\D(\c_i)$ . The rest of the proof runs as before. \
Observe that as before we can apply this argument to any idempotent $e_j$ and then $\Fr\A\geq k$.

\end{proof}

\newpage
\section{Open Frobenius Cluster-Tilted Algebras of type $\mathbb{E}$}

We start by computing the mutation class of  $\mathbb{E}_6$. This can be done, for example, by using the Java applet of
Keller~\cite{Keller-software}. The mutation class of $\mathbb{E}_6$ consists of
67 quivers. In the table below we list all the quivers in the mutation class of type $\mathbb{E}_6$.
In some quivers certain arrows are replaced by undirected lines; this has to be
read that these lines can take any orientation. For each  quiver in the table  we
compute the relations of the corresponding cluster-tilted algebra
according to~\cite{BMR}. Since there is at most one arrow between any two vertices, we indicate a path by the sequence of vertices it traverses. A zero-relation is then indicated by a sequence of the form $(a,b,c,\dots)$ and a commutativity relation has the form $(a,b,c,\dots)-(a',b',c',\dots)$.

\begin{center}
\begin{tabular}{|c|c|c|}
  \hline
    No. & Quiver     &  Relations   \\ \hline

1&  \begin{tabular}{c} \\ \SelectTips{eu}{10}\xymatrix@R=1.8pc@C=1.8pc{   && \scriptstyle{6} \ar@{-}[d] && \\ \scriptstyle{1} \ar@{-}[r] & \scriptstyle{2} \ar@{-}[r]  & \scriptstyle{3} \ar@{-}[r] & \scriptstyle{4} \ar@{-}[r] & \scriptstyle{5}  } \\ \\ \end{tabular}
& {\scriptsize\begin{tabular}{c} \\ None \\ \end{tabular}} \\ \hline

2&  \begin{tabular}{c} \\ \SelectTips{eu}{10}\xymatrix@R=1.8pc@C=1.8pc{   & \scriptstyle{3} \ar[dr] & \scriptstyle{5}\ar@{-}[r] \ar[l] & \scriptstyle{6} \\  \scriptstyle{1} \ar@{-}[r] & \scriptstyle{2} \ar[u]   & \scriptstyle{4}  \ar[u] \ar[l] &   }  \\ \\ \end{tabular}  &  {\scriptsize\begin{tabular}{c} \\ $(2,3,4)$, $(3,4,2)$, $(5,3,4)$, $(3,4,5)$,\\ $(4,2,3)-(4,5,3)$ \\
\end{tabular}} \\\hline

3& \begin{tabular}{c} \\  \SelectTips{eu}{10}\xymatrix@R=1.8pc@C=1.8pc{   & \scriptstyle{6} \ar[r]  & \scriptstyle{5}\ar[dl] & \\  \scriptstyle{1} \ar@{-}[r] & \scriptstyle{2} \ar[u] \ar[r]  & \scriptstyle{3} \ar@{-}[r] \ar[u]  &  \scriptstyle{4}   } \\ \\ \end{tabular} & {\scriptsize \begin{tabular}{c} \\ $(3,5,2)$, $(5,2,3)$, $(6,5,2)$, $(5,2,6)$, \\
$(2,6,5)-(2,3,5)$ \\ \end{tabular}} \\\hline

4& \begin{tabular}{c} \\ \SelectTips{eu}{10}\xymatrix@R=1.8pc@C=1.8pc{   & \scriptstyle{1} \ar[r]  \ar[d] & \scriptstyle{4} \ar[d] & \\  \scriptstyle{5} \ar@{-}[r] & \scriptstyle{2} \ar[r] \ar[r]  & \scriptstyle{3} \ar@{-}[r] \ar[ul]  &  \scriptstyle{6}   } \\ \\ \end{tabular} & {\scriptsize \begin{tabular}{c} \\ $(2,3,1)$, $(3,1,2)$, $(4,3,1)$, $(3,1,4)$, \\ $(1,2,3)-(1,4,3)$ \\ \end{tabular}}\\\hline

5& \begin{tabular}{c} \\ \SelectTips{eu}{10}\xymatrix@R=1.1pc@C=.55pc{   && & \scriptstyle{3} \ar[dr] &  \\  \scriptstyle{1} \ar@{-}[rr] & &\scriptstyle{2} \ar[drr] \ar[ur]    && \scriptstyle{4}  \ar[ll]  \\  && \scriptstyle{5} \ar[u] & & \scriptstyle{6}  \ar[ll] \ar[u] } \\ \\ \end{tabular} & {\scriptsize \begin{tabular}{c} \\ $(3,4,2)$, $(4,2,3)$, $(4,2,6)$, $(2,6,5)$, \\ $(5,2,6)$,
$(2,3,4)-(2,6,4)$, \\ $(6,4,2)-(6,5,2)$ \\ \end{tabular}} \\\hline

6& \begin{tabular}{c} \\ \SelectTips{eu}{10}\xymatrix@R=1.1pc@C=.55pc{ &  & & \scriptstyle{3} \ar[dr] &  \\  \scriptstyle{1} \ar@{-}[rr] && \scriptstyle{2} \ar[ur] \ar[d]   && \scriptstyle{4}  \ar[ll] \ar[d] \\ & & \scriptstyle{6} \ar[rru] & & \scriptstyle{5} \ar[ll] } \\ \\ \end{tabular} &  {\scriptsize \begin{tabular}{c}\\ $(3,4,2)$, $(4,2,3)$, $(6,4,2)$, $(5,6,4)$, \\ $(6,4,5)$,
$(2,3,4)-(2,6,4)$, \\ $(4,2,6)-(4,5,6)$ \\ \end{tabular}} \\\hline

\end{tabular}
\end{center}

\begin{center}
\begin{tabular}{|c|c|c|}
  \hline
    No. & Quiver     &  Relations   \\ \hline

7& \begin{tabular}{c} \\ \SelectTips{eu}{10}\xymatrix@R=1.1pc@C=.55pc{  & & & \scriptstyle{6} \ar[dl] &  \\  \scriptstyle{1} \ar@{-}[rr] && \scriptstyle{2} \ar[rr]    && \scriptstyle{5}  \ar[lu] \ar[dll]  \\  && \scriptstyle{3} \ar[rr] \ar[u] & & \scriptstyle{4} \ar[u] } \\ \\ \end{tabular} & {\scriptsize \begin{tabular}{c} \\  $(2,5,6)$, $(6,2,5)$, $(2,5,3)$, $(4,5,3)$, \\ $(5,3,4)$, $(5,6,2)-(5,3,2)$, \\ $(3,2,5)-(3,4,5)$ \\ \end{tabular}}  \\\hline

8& \begin{tabular}{c} \\ \SelectTips{eu}{10}\xymatrix@R=1.8pc@C=1.8pc{   \scriptstyle{1} \ar[r] & \scriptstyle{2} \ar[ld] \ar[rd]  & \scriptstyle{3} \ar[l] \\  \scriptstyle{4} \ar[u] \ar[r] & \scriptstyle{5} \ar[u] &   \scriptstyle{6} \ar[u] \ar[l]  } \\ \\ \end{tabular} &  {\scriptsize \begin{tabular}{c} \\ $(1,2,4)$, $(2,4,1)$, $(5,2,4)$, $(5,2,6)$, \\ $(2,6,3)$, $(3,2,6)$, $(4,1,2)-(4,5,2)$, \\ $(2,4,5)-(2,6,5)$, $(6,3,2)-(6,5,2)$ \\
\end{tabular}}  \\\hline

9& \begin{tabular}{c} \\ \SelectTips{eu}{10}\xymatrix@R=1.8pc@C=1.8pc{   \scriptstyle{1} \ar[r] & \scriptstyle{2} \ar[r] \ar[dl] & \scriptstyle{3} \ar[d] \\  \scriptstyle{4} \ar[u]  \ar[r] & \scriptstyle{5}  \ar[u] &   \scriptstyle{6} \ar[l] } \\ \\ \end{tabular}
& {\scriptsize \begin{tabular}{c}  \\ $(1,2,4)$, $(2,4,1)$, $(5,2,4)$, $(3,6,5,2)$, \\ $(6,5,2,3)$, $(5,2,3,6)$,
$(4,1,2)-(4,5,2)$, \\ $(2,4,5)-(2,3,6,5)$ \\ \end{tabular} } \\\hline

10& \begin{tabular}{c} \\ \SelectTips{eu}{10}\xymatrix@R=1.8pc@C=1.8pc{   \scriptstyle{2} \ar[d] & \scriptstyle{3} \ar[d] \ar[r]  \ar[l] & \scriptstyle{5} \ar[d] \\  \scriptstyle{1} \ar[ur] & \scriptstyle{4} \ar[r] \ar[l] &   \scriptstyle{6} \ar[ul] } \\ \\ \end{tabular} &  {\scriptsize \begin{tabular}{c} \\ $(1,3,2)$, $(2,1,3)$,
$(1,3,4)$, $(6,3,4)$, \\ $(6,3,5)$, $(5,6,3)$, $(3,2,1)-(3,4,1)$, \\ $(3,4,6)-(3,5,6)$, $(4,1,3)-(4,6,3)$ \\ \end{tabular} } \\\hline

11& \begin{tabular}{c} \\ \SelectTips{eu}{10}\xymatrix@R=1.8pc@C=1.8pc{   \scriptstyle{1} \ar[r] & \scriptstyle{2} \ar[r] \ar[d] & \scriptstyle{3} \ar[d] \\  \scriptstyle{4} \ar[u] & \scriptstyle{5} \ar[r]  \ar[l] &   \scriptstyle{6} \ar[ul] } \\ \\ \end{tabular} &   {\scriptsize\begin{tabular}{c} \\ $(3,6,2)$, $(6,2,3)$, $(6,2,5)$, $(1,2,5,4)$, \\ $(2,5,4,1)$,
$(4,1,2,5)$, $(2,3,6)-(2,5,6)$, \\ $(5,6,2)-(5,4,1,2)$ \\ \end{tabular}}   \\\hline

12& \begin{tabular}{c} \\ \SelectTips{eu}{10}\xymatrix@R=1.8pc@C=1.8pc{   \scriptstyle{1} \ar[r] & \scriptstyle{2}  \ar[dl] \ar[r]  & \scriptstyle{3} \ar[dl] \\  \scriptstyle{4} \ar[u]  \ar[r]& \scriptstyle{5} \ar[r] \ar[u] &   \scriptstyle{6} \ar[u] } \\ \\ \end{tabular} & {\scriptsize \begin{tabular}{c} \\ $(1,2,4)$,
$(2,4,1)$, $(5,2,4)$, $(3,5,2)$, \\ $(3,5,6)$, $(6,3,5)$, $(2,4,5)-(2,3,5)$, \\ $(4,1,2)-(4,5,2)$, $(5,2,3)-(5,6,3)$ \\ \end{tabular}} \\\hline

13& \begin{tabular}{c} \\ \SelectTips{eu}{10}\xymatrix@R=1.1pc@C=.55pc{    & & & \scriptstyle{3} \ar[dr] &  \\  \scriptstyle{1} \ar@{-}[rr] && \scriptstyle{2} \ar[ur] \ar[d]   && \scriptstyle{4}  \ar[ll]  \\  & & \scriptstyle{5} \ar[rr] & & \scriptstyle{6} \ar[u] } \\ \\ \end{tabular} &   {\scriptsize \begin{tabular}{c} \\ $(3,4,2)$, $(4,2,3)$, $(5,6,4,2)$, $(6,4,2,5)$, \\ $(4,2,5,6)$,
$(2,3,4)-(2,5,6,4)$ \\ \end{tabular}}  \\\hline

14& \begin{tabular}{c} \\ \SelectTips{eu}{10}\xymatrix@R=1.1pc@C=.55pc{  & & & \scriptstyle{6} \ar[dl] &  \\  \scriptstyle{1} \ar@{-}[rr] && \scriptstyle{2} \ar[rr]    && \scriptstyle{3}  \ar[lu] \ar[d] \\ & & \scriptstyle{5}  \ar[u] & & \scriptstyle{6} \ar[ll] } \\ \\ \end{tabular} &  {\scriptsize \begin{tabular}{c} \\ $(2,3,6)$, $(6,2,3)$, $(4,5,2,3)$, $(5,2,3,4)$, \\ $(2,3,4,5)$, $(3,6,2)-(3,4,5,2)$ \\ \end{tabular}}  \\\hline

\end{tabular}
\end{center}

\begin{center}
\begin{tabular}{|c|c|c|}
  \hline
    No. & Quiver     &  Relations   \\ \hline

15& \begin{tabular}{c} \\ \SelectTips{eu}{10}\xymatrix@R=1.1pc@C=.55pc{   && & \scriptstyle{6} \ar[dl] &  \\  && \scriptstyle{3} \ar[rr]    && \scriptstyle{4}  \ar[lu] \ar[d] \\ \scriptstyle{1} \ar@{-}[rr]  && \scriptstyle{2}  \ar[u] & & \scriptstyle{5} \ar[ll] } \\ \\ \end{tabular} &    {\scriptsize \begin{tabular}{c} \\ $(3,4,6)$, $(6,3,4)$,
$(5,2,3,4)$, $(2,3,4,5)$, \\ $(3,4,5,2)$, $(4,6,3)-(4,5,2,3)$ \\ \end{tabular} }  \\\hline

16& \begin{tabular}{c} \\ \SelectTips{eu}{10}\xymatrix@R=1.1pc@C=.55pc{   && & \scriptstyle{6} \ar[dr] &  \\   && \scriptstyle{5} \ar[ur] \ar[d]   && \scriptstyle{4}  \ar[ll]  \\  \scriptstyle{1} \ar@{-}[rr] && \scriptstyle{2} \ar[rr] & & \scriptstyle{3} \ar[u] } \\ \\ \end{tabular} & {\scriptsize \begin{tabular}{c} \\ $(4,5,6)$, $(6,4,5)$, $(2,3,4,5)$, $(3,4,5,2)$, \\ $(4,5,2,3)$, $(5,6,4)-(5,2,3,4)$ \\ \end{tabular} } \\\hline

17& \begin{tabular}{c} \\ \SelectTips{eu}{10}\xymatrix@R=1.8pc@C=1.8pc{   \scriptstyle{1} \ar[r] & \scriptstyle{2} \ar[d]  & \scriptstyle{3} \ar[l] \\  \scriptstyle{4} \ar[u] & \scriptstyle{5} \ar[l] \ar[r] &   \scriptstyle{6} \ar[u] } \\ \\ \end{tabular} &  {\scriptsize \begin{tabular}{c} \\ $(1,2,5,4)$, $(2,5,4,1)$,
$(4,1,2,5)$, \\ $(3,2,5,6)$, $(2,5,6,3)$, $(6,3,2,5)$, \\ $(5,4,1,2)-(5,6,3,2)$ \\ \end{tabular}}   \\\hline

18& \begin{tabular}{c} \\  \SelectTips{eu}{10}\xymatrix@R=1pc@C=.7pc{  &  \scriptstyle{4} \ar[dr] & \scriptstyle{1}   & \scriptstyle{6} \ar[dl] &  \\  \scriptstyle{3} \ar[ur] & & \scriptstyle{2} \ar@{-}[u]  \ar[rr]  \ar[ll] & &  \scriptstyle{5} \ar[lu] } \\ \\ \end{tabular}
& {\scriptsize\begin{tabular}{c} \\ $(2,3,4)$, $(3,4,2)$, $(4,2,3)$, \\ $(2,5,6)$, $(5,6,2)$, $(6,2,5)$ \\
\end{tabular}}  \\\hline

19& \begin{tabular}{c} \\  \SelectTips{eu}{10}\xymatrix@R=1pc@C=.55pc{  & & && \scriptstyle{6} \ar@{-}[d] &\scriptstyle{5} \ar[ld] &  \\ \scriptstyle{1} \ar@{-}[rr] & & \scriptstyle{2} \ar@{-}[rr] & & \scriptstyle{3} \ar[rr] &  & \scriptstyle{4} \ar[lu]   } \\ \\ \end{tabular} &  {\scriptsize\begin{tabular}{c} \\
$(3,4,5)$, $(4,5,3)$, $(5,3,4)$  \\ \\ \end{tabular} } \\\hline

20& \begin{tabular}{c} \\  \SelectTips{eu}{10}\xymatrix@R=1.8pc@C=1.8pc{   & \scriptstyle{5} \ar[d]  & \scriptstyle{6}\ar[l] & \\  \scriptstyle{1} \ar@{-}[r] & \scriptstyle{2}  \ar[r]  & \scriptstyle{4} \ar@{-}[r] \ar[u]  &  \scriptstyle{3}   } \\ \\ \end{tabular} &  {\scriptsize\begin{tabular}{c} \\ $(2,4,6,5)$, $(4,6,5,2)$, $(6,5,2,4)$, \\ $(5,2,4,6)$ \\
\end{tabular}} \\\hline

21& \begin{tabular}{c} \\  \SelectTips{eu}{10}\xymatrix@R=.3pc@C=1.9pc{  &  \scriptstyle{6} \ar[dd] & \scriptstyle{5}  \ar[l]   &  \\ & & & \scriptstyle{4} \ar[ul] \\ \scriptstyle{1}  \ar@{-}[r]  & \scriptstyle{2} \ar[r] & \scriptstyle{3} \ar[ru]  &    } \\ \\ \end{tabular} & {\scriptsize \begin{tabular}{c} \\ $(2,3,4,5,6)$, $(3,4,5,6,2)$, $(4,5,6,2,3)$, \\ $(5,6,2,3,4)$, $(6,2,3,4,5)$ \\
\end{tabular} } \\\hline

\end{tabular}
\end{center}

\begin{teo}
Let $\A$ be a non hereditary cluster tilted algebra of type $\mathbb{E}_6$. Then $\Fr\A\geq 1$.
\end{teo}

\begin{proof}
Observe that we are reduced to prove  that every quiver on the table above,  except the first one, has a vertex $a$ where any coproduct $\D$ satisfies $\D(e_a)\neq 0$. For quivers $2, \cdots, 12$ there is a vertex $a$ of valency $3$ of the following  type (or its dual)


 $$ \begin{array}{cc} \SelectTips{eu}{10}
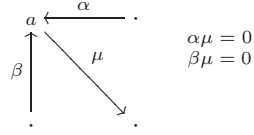
\xymatrix@R=1pc@C=1pc{    \scriptstyle{a} \ar[ddrr]^{\mu} & &\cdot \ar[ll]_{\alpha} \\
 &  & \\    \cdot \ar[uu]^{\beta}  &  & \cdot }   &   \scriptsize{ \begin{array}{c}  \\  \\ \\ \alpha\mu =0 \\  \beta\mu=0 \end{array} } \end{array} $$

Then, $\D(e_a)=A\mu\otimes \alpha + B \mu\otimes \beta +  \cdots $ where the relations $\alpha\mu =0 $ and $\beta\mu=0$ implies  $A,B\neq 0$.\\

Quivers $13, \cdots, 16$  have the following subquiver (or its dual)

 $$ \begin{array}{cc} \SelectTips{eu}{10}
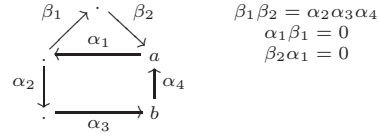
\xymatrix@R=1pc@C=1pc{    & & & \scriptstyle{.} \ar[dr]^{\beta_2} &  \\   && \scriptstyle{.} \ar[ur]^{\beta_1} \ar[d]_{\alpha_2}   && \scriptstyle{a}  \ar[ll]_{\alpha_1}  \\  & & \scriptstyle{.} \ar[rr]_{\alpha_3} & & \scriptstyle{b} \ar[u]_{\alpha_4} }  &   \scriptsize{ \begin{array}{c}  \\  \\ \\ \b_1\b_2=\a_2\a_3\a_4 \\ \a_1\b_1=0 \\ \b_2\a_1=0 \\ \end{array} } \end{array} $$

with  $a$  a vertex of valency $3$  and $b$ a vertex  of valency $2$.  Then $\D(e_a)=A\a_1\otimes \alpha_2\a_3\a_4 + \cdots $ and  $\D(e_b)=A\a_4\a_1\otimes \alpha_2\a_3 + \cdots $.  The relations on the quiver  implies  $A\neq 0$.\\

Quiver $17 $  is the following

$$ \begin{array}{cc} \SelectTips{eu}{10}
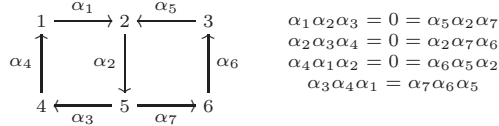
\xymatrix@R=1.8pc@C=1.8pc{   \scriptstyle{1} \ar[r]^{\a_1} & \scriptstyle{2} \ar[d]_{\a_2}  & \scriptstyle{3} \ar[l]_{\a_5} \\  \scriptstyle{4} \ar[u]^{\a_4} & \scriptstyle{5} \ar[l]^{\a_3} \ar[r]_{\a_7} &   \scriptstyle{6} \ar[u]_{\a_6} } &   \scriptsize{ \begin{array}{c}  \\  \\ \\ \a_1\a_2\a_3=0=\a_5\a_2\a_7 \\ \a_2\a_3\a_4=0=\a_2\a_7\a_6 \\ \a_4\a_1\a_2=0=\a_6\a_5\a_2 \\ \a_3\a_4\a_1=\a_7\a_6\a_5 \end{array} } \end{array} $$

Then, $\D(e_5)=A\a_7\a_6\a_5\otimes \alpha_5\a_2 + \cdots $ and $\D(e_4)=A\a_4\a_1\otimes \alpha_5\a_2\a_3 + \cdots $ with $A\neq 0$.\\

For quivers $18$ and $19$ we have the following subquiver

$$ \begin{array}{cc} \SelectTips{eu}{10}
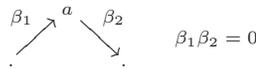
\xymatrix@R=1pc@C=1pc{    & & & \scriptstyle{a} \ar[dr]^{\beta_2} &  \\   && \scriptstyle{.} \ar[ur]^{\beta_1} && \scriptstyle{.}  }  &   \scriptsize{ \begin{array}{c}  \\  \\ \\ \b_1\b_2=0 \\ \end{array} } \end{array} $$

with  $a$  a vertex of valency $2$. Then,   the relation $\b_1\b_2=0$  implies  that the term $A\b_2\otimes \b_1$ (with $A\neq 0$) appears in the expression of $\D(e_a)$.\\

Finally, for quivers $20$ and $21$ we have a generalization of the case above where there is  a vertex $a$ of valency $2$ as above but  $\b_1\b_2\neq 0$ and there are two non-zero paths $\mu$ and $\rho$ such that  $s(\mu)=t(\rho)=a$ and $\rho\b_2=0=\b_1\mu$. Then we can guarantee that the term  $A \mu\otimes \rho$ appears in the expression of $\D(e_a)$.
\end{proof}

\end{document}